\documentclass[reqno, 11pt,a4paper]{amsart}

\usepackage{pdfsync}
\usepackage{stmaryrd} 
\usepackage{marginnote}
\usepackage[colorlinks=true, pdfstartview=FitV, linkcolor=blue,citecolor=blue, urlcolor=blue,pagebackref=false]{hyperref}
\usepackage{color}
 \definecolor{refkey}{gray}{.5}
 \definecolor{labelkey}{gray}{.5}
\definecolor{light}{gray}{.9}
\usepackage[showlabels,sections,floats,textmath,displaymath]{}
\usepackage[percent]{overpic}

\usepackage{calc,amsfonts,amsthm,amscd,epsfig,psfrag,amsmath,amssymb,enumerate,paralist,mathrsfs,mathtools,dsfont,graphicx}
\usepackage[american]{babel}
\mathtoolsset{showonlyrefs,showmanualtags}

\usepackage[showlabels,sections,floats,textmath,displaymath]{}

\reversemarginpar
\newlength\fullwidth
\numberwithin{equation}{section}

\DeclareMathSymbol{\leqslant}{\mathalpha}{AMSa}{"36} 
\DeclareMathSymbol{\geqslant}{\mathalpha}{AMSa}{"3E} 
\DeclareMathSymbol{\eset}{\mathalpha}{AMSb}{"3F}     
\renewcommand{\leq}{\;\leqslant\;}                   
\renewcommand{\geq}{\;\geqslant\;}                   
\newcommand{\suptwo}[2]{\sup_{\substack{#1 \\ #2}}} 
\newcommand{\sumtwo}[2]{\sum_{\substack{#1 \\ #2}}} 
\renewcommand{\b}{\beta}

\renewcommand{\restriction}{\mathord{\upharpoonright}}

\def\1{\ifmmode {1\hskip -3pt \rm{I}} \else {\hbox {$1\hskip -3pt \rm{I}$}}\fi}

\newcommand{\D}{\Delta}

\renewcommand{\b}{\beta}
\renewcommand{\l}{\lambda}
\renewcommand{\L}{\Lambda}
\newcommand{\h}{\eta}

\renewcommand{\l}{\lambda}

\renewcommand{\d}{\delta}
\renewcommand{\t}{\tau}

\newcommand{\g}{\gamma}
\newcommand{\G}{\Gamma}

\newcommand{\e}{\varepsilon}

\renewcommand{\O}{\Omega}

\newcommand{\tc}{\thinspace |\thinspace}

\newtheorem{theorem}{Theorem}[section]
\newtheorem{lemma}[theorem]{Lemma}
\newtheorem{proposition}[theorem]{Proposition}
\newtheorem{corollary}[theorem]{Corollary}

\newtheorem{definition}[theorem]{Definition}

\newtheorem*{question*}{Question}

\newtheorem*{remark*}{Remark}
\newtheorem*{idefinition*}{Definition}

\newcommand{\Z}{\mathbb Z}

\newcommand{\cB}{\ensuremath{\mathcal B}}
\newcommand{\cC}{\ensuremath{\mathcal C}}

\newcommand{\cH}{\ensuremath{\mathcal H}}

\newcommand{\cS}{\ensuremath{\mathcal S}}

\newcommand{\cU}{\ensuremath{\mathcal U}}

\newcommand{\cZ}{\ensuremath{\mathcal Z}}

\newcommand{\bbE}{{\ensuremath{\mathbb E}} }

\newcommand{\bbN}{{\ensuremath{\mathbb N}} }

\newcommand{\bbP}{{\ensuremath{\mathbb P}} }

\newcommand{\bbR}{{\ensuremath{\mathbb R}} }

\newcommand{\bbZ}{{\ensuremath{\mathbb Z}} }

\newcommand{\sC}{{\ensuremath{\mathscr C}}}

\newcommand{\wt}{\widetilde }

\newcommand{\gep}{\varepsilon}

\date{June 5, 2014; Revised July 28, 2015}

\title[Large deviations of the 2D SOS model]{On the probability of staying above a wall for the $(2+1)$-dimensional SOS model at low temperature}
\author{Pietro Caputo, Fabio Martinelli and Fabio Lucio Toninelli}
\address{Dipartimento di Matematica e Fisica, Universit\`a Roma
  Tre, Largo S. Murialdo 1, 00146 Roma, Italy}
\email{caputo@mat.uniroma3.it, martin@mat.uniroma3.it}

\address{Universit\'e de Lyon, CNRS and Institut Camille Jordan, Universit\'e Lyon 1,
    43 bd
 du 11 novembre 1918, 69622 Villeurbanne, France}
\email{toninelli@math.univ-lyon1.fr}

\begin{document}
\begin{abstract}
  We obtain sharp asymptotics for the probability that the
  $(2+1)$-dimensional discrete SOS interface at low temperature is
  positive in a large region. For a square region $\L$, both
  under the infinite volume measure and under the measure with zero
  boundary conditions around $\L$, this  probability turns out to behave
  like $\exp(-\tau_\beta(0) L \log L )$, with $\tau_\beta(0)$ the
  surface tension at zero tilt, also called step free energy, and $L$ the box side. This behavior is qualitatively different from
  the one found for continuous height massless gradient
  interface models \cite{BDZ,DG}. 
\end{abstract}


\maketitle

\section{Introduction}

Let $\bbP_\L$ denote the Gibbs measure of the $(2+1)$-dimensional SOS
model on a box $\L\subset \mathbb Z^2$ with zero boundary condition. The configurations
are discrete height functions $\eta:\Lambda\mapsto
\mathbb Z$ whereas
$\eta(x)=0$ for $x\notin \L$.
The probability
measure is given by
\[
\bbP_\L(\eta)=\frac{\exp{\big(-\beta\sum_{|x-y|=1}|\eta(x)-\eta(y)|\big)}}{Z_\L},
\]
where $\beta>0$ is the inverse temperature, and $Z_\L$ denotes
the associated normalizing factor, called partition function. We will mostly consider the case where
 $\L=\L_L=[-L,L]^2\cap\bbZ^2$ is the square of side $2L+1$ in $\bbZ^2$ centered at the origin. 

It is well known that, if $\beta$ is
sufficiently large (as we assume from here on), the 
limit of $\bbP_{\L_L}$ as $L\to\infty$ exists (in the sense that
the probability of any local event converges), and is denoted 
$\bbP$, the infinite-volume Gibbs measure; see e.g.\ \cite{BW}. 

The infinite volume measure is characterized by the fact that heights
have finite variance and exponentially decaying tails: the interface
is globally very rigid and flat, the height is exactly
zero on a set of sites of density $1-O(\exp(-4\beta))$ and typical fluctuations are isolated
spikes; see \cite{BW,Bricetal,CLMST2}. The question we investigate here is that of large fluctuations
of the interface, namely, the asymptotics of the probability that the
interface is positive in a fixed large region.
In order to formulate our main result, let us recall the definition of
the {\em surface tension} at zero tilt, often referred to as {\em step free energy}:
\begin{definition}\label{deftb}
Let $\xi$ be the height function on $\L_L^c$ such that
  $\xi(x)=1$ if $x=(x_1,x_2)$ with 
$x_2\geq0$,  and $\xi(x)=0$ otherwise.  
Let $Z_{\L_L}^{\xi}$ be the
partition function on $\L_L$ with boundary condition $\xi$ (see Section \ref{notation} below for more details). 
Then, the  surface tension at zero tilt is defined as
\begin{eqnarray}
  \label{eq:3}
  \tau_\b(0)=-\lim_{L\to\infty}\frac{1}{2\beta L}\log\frac{ Z_{\L_L}^{\xi}}{ Z_{\L_L}}.
\end{eqnarray}
\end{definition}
It is a known fact that $\t_\b(0)$ is well defined and that, for $\b$ sufficiently large, one has $\t_\b(0)>0$, see Lemma
\ref{lem:surftens} below for more details.
We have then:
\begin{theorem}\label{th:main}
There exists $\b_0>0$ such that for any $\b\geq \b_0$ one has
\begin{align}\label{LDsos}
\lim_{L\to\infty}\frac1{L\log L}\log \bbP_{\L_L} \left( \eta(x) \geq
0\;\text{for every\;} x\in \L_L\right) = -2\t_\b(0).
\end{align}
The same limit holds if we replace $\bbP_{\L_L}$ by $\bbP$.
\end{theorem}
Actually, it will be clear from the proof 
that the result still holds if we
replace the inequality $\eta(x) \geq
0 $ with $\eta(x)\geq n$, for any fixed $n>0$.

\smallskip

We now describe the heuristics behind Theorem \ref{th:main}. In
\cite{CLMST2} (see also \cite{CRAS} for a summary of the main results) the scaling limit of the shape of the SOS
surface in the box $\L_L$ with zero boundary conditions and 
conditioned  to be non-negative was established in full detail.
The SOS
interface lifts rigidly to a height
$H(L)=\lfloor\tfrac1{4\beta}\log L\rfloor$, in order to create room
for downward spike-like fluctuations
(entropic repulsion). As a consequence there are $H(L)$ macroscopic level lines,
following approximately $\partial\L_L$, where the height of
the surface jumps (roughly) by one. A fraction $1-o(1)$ of
the level lines is at distance $o(L)$ from $\partial \L_L$ while the rest
has a non trivial scaling limit as $L\to \infty$, with flat and curved
parts and $1/3$
fluctuation exponent along the flat part.
Roughly each of the level lines at distance $o(L)$ from $\partial\L_L$ entails
a surface energy cost $|\partial \L_L|\beta\tau_\b(0)=8\beta L \tau_\b(0)$. The
total energy cost of the macroscopic level lines ensemble is therefore 
\[
(1-o(1))8\beta \tau_\b(0)H(L) L=2(1-o(1)) \tau_\b(0) L \log L,
\]
which explains \eqref{LDsos}.
The difficulty that arises in substantiating this heuristics is that
the $H(L)$ contours have mutual interactions. If these are naively
estimated, they produce an additive term, of apriori indefinite sign, of order $O(c_\beta
|\partial\L_L|H(L))=O(\gep_\beta L\log L)$ in the energy
cost. Here $\gep_\beta=c_\b/\b>0$ is a constant tending to zero as
$\beta\to\infty$, but non-zero for any finite $\beta$. While this
problem can be avoided when
looking for   a lower bound on the l.h.s. of \eqref{LDsos},  simply by
imposing that the contours stay sufficiently far one from the other to
neglect the interaction, as an upper
bound we would get nothing better than $-2\tau_\b(0)+\gep_\beta$.

The solution we
find is an iterative monotonicity argument (Theorem \ref{th:m1}), based on the FKG
properties of the SOS model, which we
believe is of interest by itself. This allows us to conclude that the
possibly attractive effect of the mutual interaction potential is more
than compensated by the loss of entropy due to the fact that the
contours cannot mutually cross. As a consequence, the surface tension
associated to $n$ SOS contours is at least the sum of the individual
surface tensions (Corollary \ref{coro_surftens})\footnote{After completing this work we realized that a conceptually similar argument was put forward by Bricmont, El Mellouki and Fr\"ohlich \cite[Appendix 1]{Bricetal} to compare the step free energy to the free energy associated to a single macroscopic step in the boundary condition.}.

\subsection{Discussion} Since the early work of Lebowitz and Maes
\cite{LebMaes}, the problem of computing the sharp large deviation
behavior of the positivity event $\eta(x)\geq 0$, $x\in\L_L$, has
attracted much attention. Refined estimates have been obtained for
continuous height models such as the Gaussian free field on $\bbZ^d$,
see \cite{BDZ,BDG,Dzero}, as well as for more general lattice massless
free fields \cite{DG}. A large deviation theory for such models was
further developed in \cite{DGI}. The problem is of particular
relevance in the study of the entropic repulsion phenomenon
\cite{Bricetal}, see e.g.\ \cite{Yvan_survey} for a survey.  
Considerable progress has been recently made for the SOS model
\cite{CLMST,CRAS,CLMST2} and for the discrete Gaussian model
\cite{LMS} for which the SOS gradient term $|\eta(x)-\eta(y)|$ in the
energy function is replaced by
$(\eta(x)-\eta(y))^2$, but the question of computing the limit in \eqref{LDsos} remained unaddressed. 
 
As a matter of comparison, let us briefly recall the known results for the two-dimensional continuous Gaussian case.
If $\bbP_L$ denotes the 
$2$D Gaussian free field on $\L_L$ with zero boundary condition, then for any $\d\in(0,1)$
one has 
\begin{align}\label{LDGFF}
\lim_{L\to\infty}\frac1{(\log L)^2}\log \bbP_{L} \left( \eta(x) \geq
0\;\text{for every\;} x\in \L_{(1-\d)L}\right) = -\kappa(\d),
\end{align}
where $\kappa(\d)>0$ is a constant related to the relative capacity of the set $\L_{(1-\d)L}$ with respect to  $\L_L$ which satisfies $\kappa(\d)\to\infty$ as $\d\to 0$; see \cite[Theorem 3]{BDG}. On the other hand, boundary effects  dominate if {\em all} heights in $\L_L$ are required to be nonnegative, and one expects \cite[Section 3]{DG}
that 
 \begin{align}\label{LDGFF0}
\lim_{L\to\infty}\frac1{L}\log \bbP_{L} \left( \eta(x) \geq
0\;\text{for every\;} x\in \L_{L}\right) = -\chi,
\end{align}
for some $\chi>0$.  
Because of its discrete nature, the SOS interface considered in our work presents a very different behavior.
First, the rigidity of the interface allows one to consider the infinite volume limit - whereas the $2$D massless free field does not admit such a limit. Second, while the typical height in the bulk under the positivity constraint is of order $\log L$ just as in the case of the $2$D massless free field,
the cost of such a shift is much higher due to the unavoidable presence of as many as $H(L)$ macroscopic level lines each of which has a definite cost proportional to the length. In particular, boundary terms do not dominate here and the estimate of Theorem \ref{th:main} holds for $\bbP$ as well as for $\bbP_{\L_L}$.

\section{Contours, surface tension, etc.}\label{tools}
Here we define the model, and the notion of contours of the SOS interface. 
To express the law of contours we shall use a {\em cluster expansion} for partition functions of the SOS
model. Finally we recall the definition of surface tension for a
general tilt, and some of its properties.

\subsection{SOS model: basic definitions and notation}\label{notation}
We call a {\em bond} (resp. {\em
  dual bond})
any straight line segment joining two neighboring sites in ${\bbZ^2}$ (resp. of
${\bbZ^2}^*$, the dual lattice of $\bbZ^2$). Here $\bbZ^2$ and ${\bbZ^2}^*\equiv \bbZ^2+(1/2,1/2)$ are thought of as embedded in $\bbR^2$.
For any finite $\L\subset\bbZ^2$, let $\cB_\L\subset {\bbZ^2}$ denote
the set of bonds of the form $e=xy$  with $x\in\L$ and $y \in
\L\cup \partial \L$, where $\partial \L$ is the external boundary of
$\L$, i.e.\ the set of $y\in\L^c$ such that $xy$ is a bond for some $x\in\L$. 
A height configuration $\t: \L^c\mapsto \bbZ$ is called a boundary condition. 
We define $\O_\L^\t$ as the set of height functions $\eta:\bbZ^2\mapsto \bbZ$ such that $\eta(x)=\t(x)$ for all $x\notin \L$.   
The SOS Hamiltonian in $\L$ with boundary condition $\t$ is the function defined by 
\begin{align}\label{sos_ham}
\cH_\L^\t (\eta) = \sum_{xy\in\cB_\L} |\eta(x)-\eta(y)| \,,\quad \;\eta\in\O_\L^\t.
\end{align}
The SOS Gibbs measure in $\L$ with  boundary condition $\t$ at inverse temperature $\b$ is the probability measure 
$\bbP_\L^\t$ on $\O_\L^\t$ given by
\begin{align}\label{sos_gibbs}
\bbP_\L^\t(\eta) = \frac1{Z_\L^\t}\,\exp{\left(-\b\cH_\L^\t (\eta)\right)}\,,
\end{align}
where $Z_\L^\t$ is the partition function $$Z_\L^\t=\sum_{\eta\in\O_\L^\t}\exp{\left(-\b\cH_\L^\t (\eta)\right)}.$$ When $\t=0$ we simply write $Z_\L$ for $Z_\L^0$ and $\bbP_\L$ for $\bbP_\L^0$. 
We often consider boxes $\L$ of rectangular shape, and write $\L_{L,M}$, with $L,M\in \bbN$, for the rectangle 
$\L_{L,M}=([-L,L]\times [-M,M])\cap\bbZ^2$ centered at the origin. When $L=M$ we  write $\L_L$ for the square of side $2L+1$. 

We recall that the SOS model satisfies the so called FKG
inequality~\cite{FKG} with respect to the natural partial order defined by    
$\eta\leq\eta'\Leftrightarrow \eta(x)\leq \eta'(x)$ for every $x$.
That is, 
if $f$ and $g$ are two increasing (w.r.t.\ the above partial order)
functions, then
$\bbE^\t_\Lambda(fg)\ge \bbE^\t_\Lambda(f) \bbE^\t_\Lambda (g)$ for any region $\L$ and any boundary condition $\t$, where $\bbE^\t_\L$ denotes expectation w.r.t.\ $\bbP_\L^\t$. 
To prove the FKG inequality one can establish directly the validity of the FKG lattice condition 
\begin{equation}\label{holley}
\bbP_\L^{\t}(\eta \vee \eta') \bbP_\L^{\t}(\eta\wedge \eta')\ge \bbP_\L^{\t}(\eta) \bbP_\L^{\t}(\eta').
\end{equation}

\subsection{Geometric contours, $h$-contours etc.}\label{hcont}
We use the following notion of contours. 
\begin{definition}\label{contourdef}
Two sites $x,y$ in $\bbZ^2$ are said to be {\em separated by a dual bond $e$} if
their distance (in $\bbR^2$) from $e$ is $\tfrac12$. A pair of
orthogonal dual
bonds which meet in a site $x^*\in {\Z^2}^*$ is said to be a
{\sl linked pair of bonds} if both are on the same side of the
forty-five degrees line (w.r.t. to the horizontal axis) across $x^*$. A {\sl geometric contour} (for
short a contour in the sequel) is a
sequence $e_0,\ldots,e_n$ of dual bonds such that:
\begin{enumerate}
\item $e_i\ne e_j$ for $i\ne j$, except for $i=0$ and $j=n$ where $e_0=e_n$.
\item for every $i$, $e_i$ and $e_{i+1}$ have a common vertex in ${\Z^2}^*$.
\item if $e_i,e_{i+1},e_j,e_{j+1}$ all have a common vertex $x^*\in {\Z^2}^*$,
then $e_i,e_{i+1}$ and $e_j,e_{j+1}$ are linked pairs of bonds.
\end{enumerate}
We denote the length of a contour $\gamma$, i.e.  the number of distinct bonds in $\g$, by $|\gamma|$, its interior
(the sites in $\bbZ^2$ it surrounds) by $\Lambda_\gamma$ and its
interior area (the number of such sites) by
$|\Lambda_\gamma|$. Moreover we let $\Delta_{\gamma}$ be the set of sites in $\Z^2$ such that either their distance
(in $\bbR^2$) from $\gamma$ is $\tfrac12$, or their distance from the set
of vertices in ${\Z^2}^*$ where two non-linked bonds of $\gamma$ meet
equals $1/\sqrt2$. Finally we let $\Delta^+_\gamma=\Delta_\gamma\cap
\L_\g$ and $\Delta^-_\gamma = \Delta_\gamma\setminus \Delta^+_\gamma$.
Given a contour $\gamma$ we say that $\gamma$ is an \emph{$h$-contour}
for the configuration $\eta$ if
\[
\eta\restriction_{\Delta^-_\gamma}\leq h-1, \quad \eta\restriction_{\Delta^+_\gamma}\geq h.
\]
Finally $\sC_{\gamma,h}$ will denote the event that $\gamma$ is an $h$-contour.
\end{definition}

To illustrate the above definitions with a simple example, consider the elementary contour given by the square of side $1$ surrounding a site $x\in\bbZ^2$. In this case, $\g$ is an $h$-contour iff $\eta({x})\geq h$ and $\eta(y)\leq h-1$ for all $y\in\{x\pm e_1, x\pm e_2, x+ e_1+e_2,x-e_1-e_2\}$. 
We observe that 
a geometric contour $\g$ could be at the same time a
$h$-contour and a $h'$-contour with $h\neq h'$. More generally two
geometric contours $\g,\g'$ could be contours for the same surface
with different height parameters even if $\g\cap\g'\neq\emptyset$, but then
the interior of one one of them must be contained in the interior of the other; see Figure \ref{fig:002} for an example. 

\begin{figure}[htb]
        \centering
 \begin{overpic}[scale=0.37]{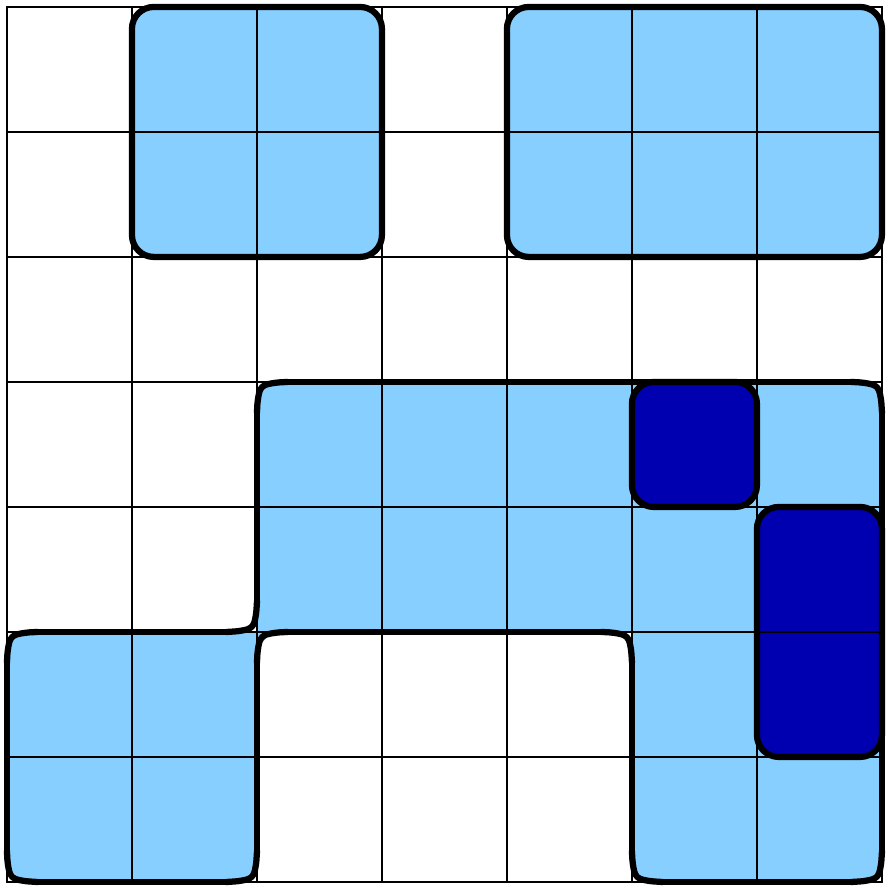}
\end{overpic}
        \caption{
        Example of a SOS configuration above the wall in the $7\times 7$ box $\L_3$ with zero boundary conditions: white sites have height $0$, shaded sites have height $1$ and darker sites have height $2$. Notice that according to Definition \ref{contourdef} there are three $1$-contours and two  $2$-contours.
          }\label{fig:002}
\end{figure}

\subsection{Cluster expansion}\label{cl_exp}
So called cluster expansions are a well established tool for the analysis of random interfaces at low-temperature; see e.g.\ \cite{BW} where both the SOS model and the discrete gaussian model are considered. Here we shall need a particular expansion that allows us to take into account the extra constraints that appear naturally
in the partition function of the SOS model on a region $\L$ delimited by two contours; see 
   \eqref{ap20} below.

Given a finite connected set $\L\subset \bbZ^2$, let 
$\partial _*\L$ denote the set of  $y\in \L$ either at distance $1$ from $\partial \L$
or at distance $\sqrt 2$ from $\partial \L$ in the south-west or north-east direction. 
Fix $U_+,U_-\subset \partial_* \L$, and let 
$ Z_{\L,U_+,U_-}$ denote the SOS partition function in $\L$ with the sum over $\eta$  restricted to those
$\eta\in \O_\L^0$ such that $\eta(x)\geq 0$ for all $x\in U_+$ and $\eta(x)\leq 0$ for all $x\in U_-$. 
Clearly, if $U_-\cap U_+\neq \emptyset$, then $\eta(x)=0$ is fixed for
all $x\in U_-\cap U_+$. If $\L=\emptyset$ then $Z_{\L,U_+,U_-}:= 1$. 
We refer the reader to \cite[App. A]{CLMST} for a proof of the following statement based on the general cluster expansion from \cite{KotPre}.
\begin{lemma}
\label{lem:dks}
There exists $\b_0>0$ independent of $\L$ such that for all $\b\geq \b_0$, for all finite connected $\L\subset\bbZ^2$ and $U_+,U_-\subset \partial_* \L$:
\begin{equation}\label{zetaexp}
\log  Z_{\L,U_+,U_-}=\sum_{V\subset \L}\varphi_{U_+,U_-}(V),
\end{equation}
where the potentials $\varphi_{U_+,U_-}(V)$ satisfy
\begin{enumerate}[(i)]
\item $\varphi_{U_+,U_-}(V)=0$ if $V$ is not connected.
\item $\varphi_{U_+,U_-}(V)=\varphi_0(V)$ if $V\cap({U_+\cup U_-})=\emptyset$, for
  some shift invariant  potential $V\mapsto \varphi_0(V)$,
that is
\[
\varphi_0(V)=
\varphi_0(V+x)\quad \forall \,x\in\bbZ^2,
\]
where $\varphi_0$ is independent of $U_+,U_-,\L$.
\item For all $V\subset \L$:
\[
\sup_{{U_+,U_-}\subset\partial_* \L}|\varphi_{U_+,U_-}(V)|\leq
  \exp(-(\b-\b_0)\, d(V))
\]
where $d(V)$ is the cardinality of the smallest connected set of all dual
bonds separating points of $V$ from points of its
complement (a dual bond separates $V$ from $V^c$ iff it is orthogonal to a bond connecting $V$ to $V^c$). 
\end{enumerate}
\end{lemma}

\subsection{Nested contours}\label{wedding}
Consider the rectangle $\L_{L,M}$,  for some $L,M\in \bbN$, 
and let $\bbP_{\L}$ denote the $SOS$ Gibbs measure in $\L:=\L_{L,M}$ with zero boundary conditions. Given two contours $\g,\g'$, we write $\g\subset \g'$ if $\L_\g\subset\L_{\g'}$. Fix $n\in\bbN$ and pick $n$ geometric contours $\g_1,\dots,\g_n$ such that 
$ \g_{i+1}\subset \g_i$, for every $i=1,\dots,n-1$. 
Consider the event $\cap_{i=1}^n\sC_{\g_i,i}$ that $\g_i$ is an
$i$-contour for all $i=1,\dots,n$. The probability of this event under $\bbP_{\L}$ can be expressed as
\begin{align}\label{wedprob}
\bbP_{\L}\big(\cap_{i=1}^n\sC_{\g_i,i}\big) = \frac{Z(\g_1,\dots,\g_n; L,M)}{Z_\L},
\end{align}
where $Z_\L$ denotes the partition function of the $SOS$ model in $\L=\L_{L,M}$ with zero boundary conditions and $ Z(\g_1,\dots,\g_n; L,M)$ stands for the same summation restricted to the configurations $\eta\in\O_\L^0$ such that $\g_i$ is an $i$-contour for each $i=1,\dots,n$. 
By applying the  cluster expansion in Lemma \ref{lem:dks}, with $\L=\L_{L,M}$ and $U_\pm=\emptyset$, we can write
\begin{align}\label{ub202}
Z_\L = \textstyle \exp{\big(\sum_{V\subset\L}\varphi_0(V)\big)}.
\end{align}
To expand the partition function $Z(\g_1,\dots,\g_n; L,M)$, define $S_i:=\L_{\g_{i-1}}\setminus \L_{\g_i}$, for $i=1,\dots,n+1$, where $\L_{\g_{0}}=\L$ and $\L_{\g_{n+1}} =\emptyset$, and set 
$\D_i^+=S_i\cap \D^+_{\g_{i-1}}$, and $\D_i^-=S_i\cap \D^-_{\g_{i}}$, with the understanding that $\D_1^+ = \D_{n+1}^- =\emptyset$. Notice that $\D_i^\pm\subset \partial_*S_i$. 
Using the notation of Lemma \ref{lem:dks} a direct computation proves that 
\begin{align}\label{ap20}
Z(\g_1,\dots,\g_n; L,M)= \textstyle \exp{\big(-\b\sum_{i=1}^n|\g_i|\big)}\prod_{i=1}^{n+1}Z_{S_i,\D_i^+,\D_i^-}\,.
\end{align}
The term $\sum_{i=1}^n|\g_i|$
accounts for the minimal energy associated to the given contours. The
fact that the surface gradient across a contour $\g_i$
must be at least $1$ is encoded by the constraints on $\D_i^+,\D_i^-$
appearing in $Z_{S_i,\D_i^+,\D_i^-}$.

Therefore, the expansion \eqref{zetaexp} implies
\begin{align}\label{ap200}
Z(\g_1,\dots,\g_n; L,M)= \textstyle \exp{\big(-\b\sum_{i=1}^{n}|\g_i|+ \sum_{i=1}^{n+1}\sum_{V\subset S_i}\varphi_{\D_i^+,\D_i^-}(V)\big)}.
\end{align}
The ratio \eqref{wedprob} then becomes
\begin{align}\label{ub2a1}
\bbP_{\L}\big(\cap_{i=1}^n\sC_{\g_i,i}\big) = 
\textstyle \exp{\big(-\b\sum_{i=1}^n|\g_i| + 
 \Psi_\L(\g_1,\dots,\g_n)\big)},
\end{align}
where 
\begin{align}\label{ub211}
\Psi_\L(\g_1,\dots,\g_n) = 
\sum_{i=1}^{n+1}\!\!\!\!\sumtwo{V\subset S_i:}{ V\cap (\D_i^+\cup\D_i^-)  \neq \emptyset}\!\!\!\!\!\!(\varphi_{\D_i^+,\D_i^-}(V)-\varphi_0(V)) -\!\!\!\! \sumtwo{V\subset\L:}{V\cap (\cup_{i=1}^n\g_i) \neq \emptyset} \varphi_0(V),
\end{align}
where the condition $V\cap (\cup_{i=1}^n\g_i) \neq \emptyset$ means that $V$ intersects more than just one $S_i$. 
When $n=1$, we have only one contour $\g_1=\g$ and we define 
\begin{align}\label{ub1c}
\psi_\L(\g):=\Psi_\L(\g_1) .
\end{align}
Observe that property (iii) of the potentials $\varphi_{U_+,U_-}(V)$ in Lemma \ref{lem:dks} implies in particular that
\begin{equation}
\label{rev1}
|\psi_\L(\g)|\leq 3\sum_{V\cap(\D^+_\g\cup\D^-_\g)\neq \emptyset}e^{-2(\b-\b_0)d(V)} \leq \e(\beta)|\g|,   
\end{equation}
where $\lim_{\b\to \infty}\e(\b)=0$. Later on it will be convenient to introduce the quantity $\psi_\infty(\g)$
defined as $\psi_\L(\g)$ but without the restriction $V\subset \L$, i.e. now $S_1= \bbZ^d\setminus\L_\g$.   
\subsection{The staircase ensemble}\label{stair}
Consider the rectangle $\L_{L,M}$, for some $L,M\in \bbN$. 
Fix $n\in\bbN$ and integers 
\begin{equation}
\label{ref2}
-M\leq a_1\leq \cdots\leq a_n\leq M,\;\;\text{ and}\;\; -M\leq b_1\leq \cdots\leq b_n\leq M,  
\end{equation}
and set $a_0=b_0=-(M+1)$ and $a_{n+1}=b_{n+1}=M+1$. We define a ``staircase'' height $\t$ at the external boundary $\partial \L_{L,M}$ of our rectangle which, starting from 
height zero at the base of the rectangle (i.e. the set $(x,-(M+1)), x=-L,\dots,L$)  jumps by one at the locations specified by the two
$n$-tuples $\{a_i,b_i\}$ until it reaches height $n$:
\begin{align}\label{stairtau}
\tau(u,v)=
  \begin{cases}
  i & \ \text{if $u=-L-1$ and $a_i \leq v<a_{i+1}$ or $u=L+1$ and $b_i \leq v<b_{i+1}$, }\\
 0 & \  \text{if $u\in[-L,L]$ and $v=-M-1$}\\
n & \ \text{if $u\in[-L,L]$ and $v=M+1$},
  \end{cases}
\end{align}
where $i\in\{0,\dots,n\}$, see Figure \ref{fig:1}.
Note that if two or more values of the $a_i$ or $b_i$ coincide then the boundary height $\t$ takes jumps
higher than $1$ at those points.

Next, let $Z(a_1,\dots,a_n; b_1,\dots, b_n; L,M)$ denote the partition function of the $SOS$ model in $\L_{L,M}$ with boundary condition $\t$ as in \eqref{stairtau}. 
\begin{figure}[htb]
        \centering
 \begin{overpic}[scale=0.27]{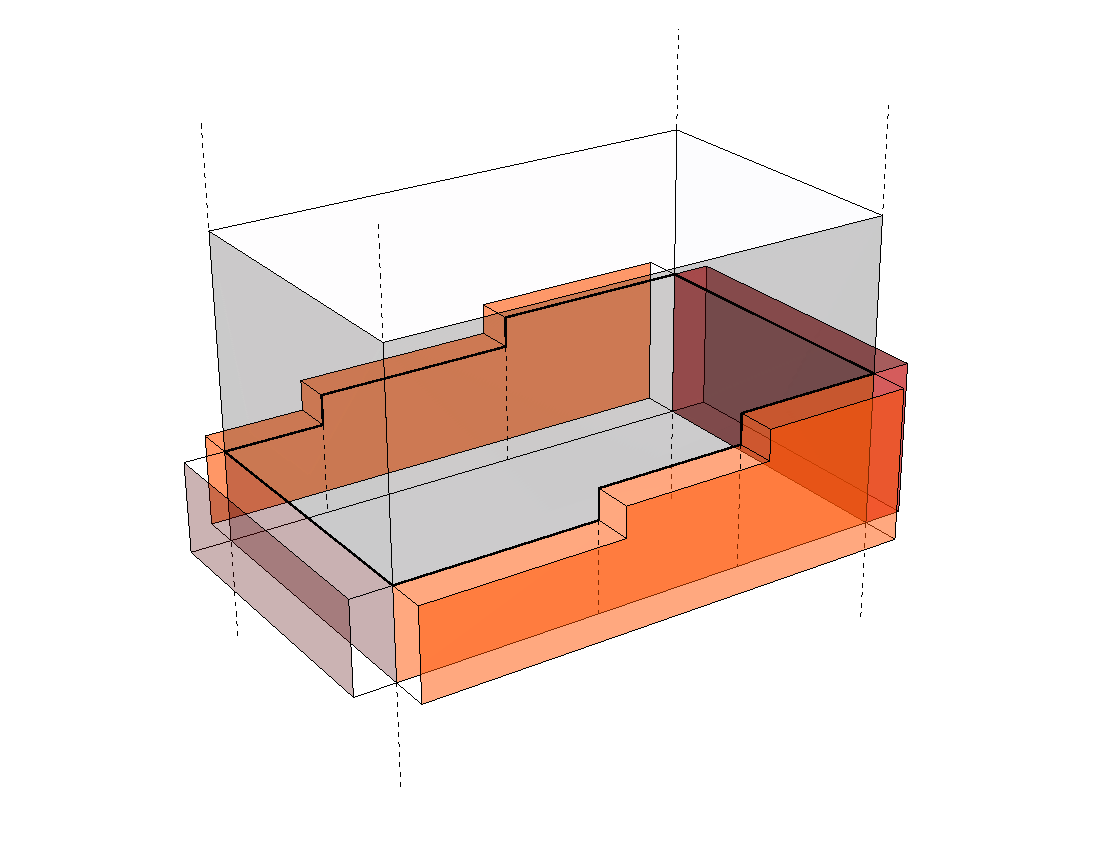}
\put(30.5,29){$z_1$}\put(30.5,31){$\cdot$}
\put(47,33.5){$z_2$}\put(47,35.8){$\cdot$}
\put(56,17){$z'_1$}
\put(68,21){$z'_2$}
\put(68.5,23.7){$\cdot$}
\put(56,19.5){$\cdot$}
\end{overpic}
        \caption{A sketch of the staircase boundary condition (\ref{stairtau}) in the rectangle $\L_{L,M}$ 
          for $n=2$. The points $z_i,z'_i$ have coordinates $z_i=(-L-1,a_i)$, and  $z'_i=(L+1,b_i)$.}
          \label{fig:1}
\end{figure}
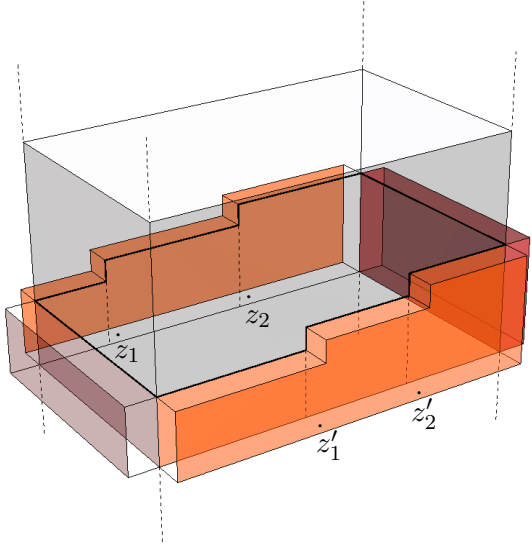

Let also $Z_\L$ denote as above the partition function of the $SOS$ model in $\L=\L_{L,M}$ with zero boundary condition everywhere. We want to compute the ratio
\begin{align}\label{rat}
\frac{Z(a_1,\dots,a_n; b_1,\dots, b_n; L,M)}{Z_\L}.
\end{align}
To expand the partition function in the numerator of \eqref{rat}, we
need the notion of an {\em open contour}. This is defined as in
Definition \ref{contourdef} except that $e_0\neq e_n$. Since the boundary conditions force the surface height to grow from height zero at the bottom base of $\L_{L,M}$ to height $n$ at the top base, necessarily any configuration $\eta$ of the $SOS$ interface compatible with $\tau$ must satisfy the following property. 

Given $\eta$ there exist uniquely defined non-crossing open contours
$\g_i$, $i=1,\dots, n$, joining the dual lattice points
$x_i:=(-L-1/2,a_i-1/2)$ and $ y_i:=(L+1/2,b_i-1/2)$ such that
$\eta(x)\leq i-1$ for all $x\in\D_i^-$ and $\eta(x)\geq i-1$ for all
$x\in\D_i^+$ where $\D_i^\pm$ are now the sets defined as follows. Let
$S_i\subset \L_{L,M}$ denote the region bounded by $\g_i$ and
$\g_{i-1}$, where $\g_{n+1}$ is the top boundary of $\L_{L,M}$ and
$ \g_{0}$ is the bottom boundary of $\L_{L,M}$.  Then $\D_i^-$ is
defined as the set of $x\in S_i$ such that either their distance from
$\g_{i}$ is $\tfrac12$, or their distance from the set of vertices in
${\Z^2}^*$ where two non-linked bonds of $\g_{i}$ meet equals
$1/\sqrt2$. Similarly, $\D_i^+$ is the set of $x\in S_i$ such that
either their distance from $\g_{i-1}$ is $\tfrac12$, or their distance
from the set of vertices in ${\Z^2}^*$ where two non-linked bonds of
$\g_{i-1}$ meet equals $1/\sqrt2$. Lemma \ref{lem:dks}
here implies
\begin{align}\label{app200}
&Z(a_1,\dots,a_n; b_1,\dots, b_n; L,M)  \nonumber
\\
 &\qquad=\textstyle \sum_{\g_1,\dots,\g_n}\exp{\big(-\b\sum_{i=1}^{n}|\g_i|+ \sum_{i=1}^{n+1}\sum_{V\subset S_i}\varphi_{\D_i^+,\D_i^-}(V)\big)},
\end{align}
where the sum ranges over all possible values of the open contours $\g_i:x_i\to y_i$ inside $\L_{L,M}$ with the non-crossing constraints. 
Recalling that $Z_\L$ can be expanded as in \eqref{ub202}, one finds that 
\begin{align}\label{ub3a1}
 &\frac{Z(a_1,\dots,a_n; b_1,\dots, b_n; L,M)}{Z_\L}\nonumber
\\
 &\qquad=
\textstyle\sum_{\g_1,\dots,\g_n}\exp{\big(-\b\sum_{i=1}^n|\g_i| + 
 \Phi_{L,M}(\g_1,\dots,\g_n)\big)},
\end{align}
where 
\begin{align}\label{ub311}
\Phi_{L,M}(\g_1,\dots,\g_n) = 
\sum_{i=1}^{n+1}\!\!\!\!\sumtwo{V\subset S_i:}{ V\cap (\D_i^+\cup\D_i^-)  \neq \emptyset}\!\!\!\!\!\!(\varphi_{\D_i^+,\D_i^-}(V)-\varphi_0(V)) - \!\!\!\!\!\!\!\!\sumtwo{V\subset \L_{L,M}:}{V\cap (\cup_{i=1}^n\g_i) \neq \emptyset} \varphi_0(V),
\end{align}
where the condition $V\cap (\cup_{i=1}^N\g_i) \neq \emptyset$ meansthat $V$ intersects more than just one $S_i$. 
Equation \eqref{ub3a1} expresses the ratio \eqref{rat} as the partition function of a gas of 
$n$ interacting non-crossing open contours $\g_1,\dots,\g_n$ within $\L_{L,M}$
such that $\g_i:x_i\to y_i$, $i=1,\dots,n$.
 Using (iii) in Lemma \ref{lem:dks} the limit as $M\to\infty$ of the above expression is well defined, so that the following holds.
 \begin{lemma}\label{lem:infvol}
For any integers $n$, $\{a_i,b_i\}_{i=1}^n$ satisfying \eqref{ref2}, the limit   \begin{equation}
    \label{eq:1}
\cZ(a_1,\dots,a_n;\, b_1,\dots,b_n;\, 
  L) :=    \lim_{M\to \infty}\frac{Z(a_1,\dots,a_n;\, b_1,\dots,b_n;\, 
  L,M)}{Z_\L}
  \end{equation}
 exists and it satisfies
 \begin{align}\label{ub4a1}
 &\cZ(a_1,\dots,a_n;\, b_1,\dots,b_n;\, 
  L) = \nonumber
\\
 &\qquad=
\textstyle\sum_{\g_1,\dots,\g_n}\exp{\big(-\b\sum_{i=1}^n|\g_i| + 
 \Phi_{L,\infty}(\g_1,\dots,\g_n)\big)},
\end{align}
where the sum ranges over all possible values of the open contours in the strip $\L_{L,\infty}:=[-L,L]\times \bbZ$
and $\Phi_{L,\infty}$ is defined as in \eqref{ub311} with $\L_{L,M}$ replaced by $\L_{L,\infty}$.
 \end{lemma}
 \begin{proof}
  Using (iii) in Lemma \ref{lem:dks} it is immediate to check that for any family of contours $(\g_1,\dots,\g_n)$ 
\[
  \lim_{M\to \infty} \Phi_{L,M}(\g_1,\dots,\g_n)= \Phi_{L,\infty}(\g_1,\dots,\g_n).
\]
As in \eqref{rev1} we have 
\begin{equation}
\label{rev4}
|\Phi_{L,M}(\g_1,\dots,\g_n) |\leq \e(\beta)\sum_{i=1}^n|\g_i|.  
\end{equation}
Hence, for $\beta $ large enough, the conclusion follows by dominated convergence.
 \end{proof}
\subsection{Surface tension}\label{surftens}
Here we recall the definition and some properties of the surface tension corresponding to arbitrary tilt. 
Let us rewrite \eqref{ub4a1} in the case $n=1$ as 
\begin{align}\label{contourmod}
 &\cZ(a_1;b_1; 
  L) =
\sum_{\g}\exp{\big(-\b|\g| + 
 \Phi_{L,\infty}(\g)\big)},
\end{align}
where the sum ranges over all open contours in the strip $\L_{L,\infty}$
joining the dual lattice points
$x_1:=(-L-1/2,a_1-1/2)$ and $ y_1:=(L+1/2,b_1-1/2)$.

\begin{lemma}\label{lem:surftens}
There exists $\b_0>0$ such that the following holds for all $\b\geq \b_0$. Let $\cZ(a_1;\, b_1;\,   L)$, denote the partition function \eqref{contourmod}.
Assume that as $L\to\infty$ one has $(b_1-a_1)/(2L)\to \l\in\bbR$ and set $\theta=\tan^{-1}(\l)$.   Then the function
\begin{equation}
    \label{eq:surftens1}
\t_\b(\theta)=-\lim_{L\to \infty}\frac{\cos(\theta)}{2\b L}\log \cZ(a_1;\,b_1;\,   L),
  \end{equation}
 is well defined and positive in $(-\pi/2,\pi/2)$. It is convex in the following
 sense: defining, for $x\in\mathbb R^2$,
 $\tau_\beta(x)=\|x\|\tau_\beta(\theta_x)$ with $\theta_x$ the  angle
 formed by the vector $x$ with the horizontal axis, $\tau_\beta$ is a
 convex function on $\mathbb R^2$.
 Moreover, 
 \begin{equation}
    \label{eq:surftens2}
\limsup_{L\to \infty}\frac1{2\b L}\sup_{a_1,b_1}\log \cZ(a_1;\,b_1;\,   L)\leq -\t_\b(0),
  \end{equation}

 \end{lemma} 
\begin{proof} 
Existence and the stated properties of the surface tension are
known facts \cite[Section 4.16]{DKS}. It is also known, see  \cite[Section 4.20]{DKS}, that $\tau_\beta(\theta)$
tends to $|\cos \theta\,|+ |\sin \theta\,|$, as $\beta\to\infty$. In particular, $\t_\b(0)\to 1$, $\beta\to\infty$. Strictly speaking the proofs in \cite{DKS} are carried out for the contour ensemble associated to the 2D Ising model, which has the form \eqref{contourmod} but with slightly different potentials in the ``decoration term'' $\Phi_{L,\infty}(\g)$. However, thanks to the properties listed in Lemma \ref{lem:dks}, the same proofs actually apply to our model in \eqref{contourmod} as well.

To prove \eqref{eq:surftens2} we distinguish two cases. If $|b_1-a_1|>4L $ we use again \eqref{rev1} to obtain
$$\suptwo{a_1,b_1:}{|b_1-a_1|>4L}\cZ(a_1;\,b_1;\,   L)\leq
\sum_{\g_1: |\g_1|\ge 5L} e^{-(\b-\e(\b))|\g_1|},$$
where $\g_1$ is a contour from $(-(L+1,a_1)$ to $((L+1),b_1)$. Clearly the above sum is negligible w.r.t.\ $\exp(-2\beta
L\tau_\beta(0))$ as $L\to \infty$ for $\b$ large enough. If instead $|b_1-a_1|\leq 4L$, then the  estimate \cite[Eq.\ (4.12.3)]{DKS} 
together with convexity of the surface tension allows one to conclude \eqref{eq:surftens2}. 
\end{proof} 
It is not hard to check that the special case $\theta=0$ coincides with the quantity in Definition \ref{deftb}.
Indeed, using the notation from Definition \ref{deftb} together with \eqref{ub3a1} (with $n=1$ and $a_1=b_1=0$),
\[
\frac{Z^\xi_{\L_L}}{Z_{\L_L}}= \frac{Z(0;0;L,L)}{Z_{\L_L}}=\textstyle\sum_{\g_1}\exp{\big(-\b|\g_1| + 
 \Phi_{L,L}(\g_1)\big)}.
\]
The same arguments of Lemma \ref{lem:infvol} can be used to check that
\[
\lim_{L\to \infty} \frac 1L \log\frac{Z^\xi_{\L_L}}{Z_{\L_L}}=\lim_{L\to \infty} \frac 1L \log\cZ(0;0;L).
\]


\section{Lower  bound}\label{lb}
Here we prove the lower bound in Theorem \ref{th:main}. 
We first establish a lower bound on the probability of having zero height at the boundary of a square.
\begin{lemma}\label{lem:bd}
For $\b\geq \b_0$ there exists $c_\b>0$ such that for any $L\in\bbN$:
\begin{align}\label{posp}
\bbP({\eta_{\partial\L_{L}}}=0)\geq e^{-c_\b L}.
\end{align}
\end{lemma}
\begin{proof}
Recall that $\bbP(\cdot)=\lim_{K\to\infty}\bbP_{\L_K}(\cdot)$. Expanding as in  \eqref{ub202}, we see that
\begin{eqnarray}
  \label{eq:4}
\textstyle  \bbP_{\L_K}({\eta_{\partial\L_{L}}}=0)=\frac{Z_{\L_K\setminus \partial\L_L}}{Z_{\L_K}}=\exp\left({-\sum_{V\subset
      \L_K,V\cap \partial \L_L\ne\emptyset}\varphi_0(V)}\right),
\end{eqnarray}
where $V\cap \partial \L_L\ne\emptyset$ is equivalent to $V$ not
contained in $\L_K\setminus \partial \L_L$. From the decay properties
of the potentials $\varphi_0$  stated in Lemma \ref{lem:dks}, the desired
result follows.
\end{proof}

\subsection{Proof of the lower bound in Theorem \ref{th:main}}
If we prove the lower bound for $\bbP_{\L_L}$ in \eqref{LDsos} 
we also have the same lower bound for $\bbP$ by using Lemma \ref{lem:bd} and  
\begin{align}\label{posp1}
\bbP( \eta_{\L_L}\geq 0) \geq   \bbP({\eta_{\partial\L_{L}}}=0)\bbP_{\L_L} ( \eta_{\L_L}\geq 0).
\end{align} 
To prove the lower bound for $\bbP_{\L_L}$ we proceed by restricting the set of configurations 
to an event $E$ defined as follows.
Fix $ N:=H(L)=\lfloor\tfrac1{4\b} \log L\rfloor$. 
Define the nested annular regions 
$\bar \cU_i:=\L_{L-3\ell_{i-1}}\setminus\L_{L-3\ell_i}$, $i=1,\dots, N$, where $\ell_0=0$ and $\ell_i=i(i+1)/2$. Notice that each $\bar \cU_i$ consists of 3 nested disjoint annuli each of width $i$. We define $\cU_i$ as the middle one, i.e.\ $\cU_i =  \L_{L-(3\ell_{i-1} -i)}\setminus\L_{L-(3\ell_i+i)}$. These sets are such that $d(\cU_i,\cU_{i+1})\geq 2i +1$, where $d(\cdot,\cdot)$ stands for the euclidean distance. 

For each $i$, define the set $\cC_i$ of all contours $\g$ 
such that $\g\subset\cU_i$ and $\g_i$ surrounds $\L_{L-(3\ell_i+i)}$. 
We consider the event $E$ that for each $i=1,\dots,N$ there exists an $i$-contour $\g_i\in\cC_i$:
\begin{align}\label{eventoE}
E = \cup_{\g_1\in\cC_1,\dots,\g_N\in\cC_N} \sC_{\g_1,1}\cap\cdots\cap\sC_{\g_N,N}.
\end{align} 
 For a fixed choice of $\g_i\subset \cC_i$, $i=1,\dots, N$ we write $S_i=\L_{\g_{i-1}}\setminus \L_{\g_i}$, and 
 $\D_i^+=S_i\cap \D^+_{\g_{i-1}}$, and $\D_i^-=S_i\cap \D^-_{\g_{i}}$ as in Section \ref{wedding}. We define  $Z_{S_i,\D_i^+,\D_i^-}^+$ as the partition function in $S_i$ with boundary conditions $i-1$ in $\partial S_i$, and with the following constraints: $\eta(x)\leq i-1$ for $x\in \D_i^+$, $\eta(x)\geq i-1$ for $x\in \D_i^-$ and $\eta(x)\geq 0$ for all $x\in S_i$. Then 
\begin{align}\label{lb4}
\bbP_{\L_L}( \eta_{\L_L}\geq 0;\sC_{\g_{1},1}\cap\cdots\cap\sC_{\g_N,N})
=\frac{e^{-\b\sum_{i=1}^N|\g_i|}\prod_{i=1}^{N+1}Z_{S_i,\D_i^+,\D_i^-}^+}{Z_\L}.
\end{align}
Below, we shall take $n:=\lfloor\log \log L\rfloor$ and fix arbitrary contours $\g^*_1\in\cC_1,\dots,\g^*_n\in\cC_n$, and sum over  the remaining contours $\g_i$, $i=n+1,\dots,N$ 
 \begin{lemma}\label{lemlb}
 Fix $\b\geq \b_0$ and fix  $\g^*_1\in\cC_1,\dots,\g^*_n\in\cC_n$, where $n=\lfloor\log \log L\rfloor$. Then
  \begin{align}\label{lb1}
 \bbP_{\L_L}( \eta_{\L_L}\geq 0; E) \geq\; \tfrac12\!\!\!\!\!\!\!\!\!\!\!\!\!\sum_{\g_{n+1}\in\cC_{n+1},\dots,\g_N\in\cC_N}\!\!\!\!\!\!\!\!\!\!\!\!
 \bbP_{\L_L}( \eta_{\L_L}\geq 0\,;\,\cap_{k=1}^n\sC_{\g^*_{k},k}\,;\,\cap_{j=n+1}^N\sC_{\g_{j},j}
 ).
 \end{align}
 \end{lemma}
\begin{proof}
Let $F_i$ denote the event that there is more than 
one $i$-contour in $\cC_i$.
Then 
\begin{align}\label{lb2}
\bbP_{\L_L}( \eta_{\L_L}\geq 0; E)  \geq
\!\!\!\!\!\!\!\!\!\sum_{\g_{n+1}\in\cC_{n+1},\dots,\g_N\in\cC_N}\!\!\!\!\!\!\!
 \bbP_{\L_L}( \eta_{\L_L}\geq 0\,;\,\cap_{k=1}^n\sC_{\g^*_{k},k}\,;\,\cap_{j=n+1}^N\sC_{\g_{j},j}
 \,;\,\cap_{i=n+1}^NF_i^c).
\end{align}
Thus, it suffices to show that for any fixed choice of $\g^*_k\in\cC_k$, $k=1,\dots,n$ and $\g_j\subset \cC_j$, $j=n+1,\dots, N$:
\begin{align}\label{lb3}
& \bbP_{\L_L}( \eta_{\L_L}\geq 0\,;\,\cap_{k=1}^n\sC_{\g^*_{k},k}\,;\,\cap_{j=n+1}^N\sC_{\g_{j},j}\,;\,\cup_{i=n+1}^NF_i)\nonumber 
\\ & \qquad \qquad\qquad \qquad \leq\tfrac12\, \bbP_{\L_L}( \eta_{\L_L}\geq 0\,;\,\cap_{k=1}^n\sC_{\g^*_{k},k}\,;\,\cap_{j=n+1}^N\sC_{\g_{j},j}).
\end{align}
Suppose the $j$-contour $\g_j\in\cC_j$ is given for each  $j=n+1,\dots,N$. If $F_i$ 
occurs then there must be a $i$-contour $\g\in\cC_i$, $\g\neq\g_i$, such that either $\g\subset S_{i+1}$ or 
$\g\subset S_{i}$. In particular, if $\cup_{i=n+1}^NF_i$ occurs, then,  for some $i\in[ n+1,N+1]$,  
there exists either an $(i-1)$-contour or an $i$-contour $\g$
inside $S_i$ and surrounding $\L_{L-(3\ell_i+i)}$. Let $\pi_{S_i,\D_i^+,\D_i^-}^+$ denote the probability measure corresponding to the 
partition function $Z_{S_i,\D_i^+,\D_i^-}^+$. From \cite[Proposition 2.7]{CLMST2} one has
that for any fixed contour $\g$ inside $S_i$, for any $h\in\bbN$:
\begin{align}\label{lb5}
\pi_{S_i,\D_i^+,\D_i^-}^+(\sC_{\g,h}) \leq \exp{\big(-\b|\g| + Ce^{-4\b h}|S_i| + Ce^{-4\b h}|\g|\log|\g|
\big)}.
\end{align}
Here and below, by $C$ we mean a positive constant that does not depend on $\b$ and $L$, whose value may change at each occurence.  Since $|S_i| \leq CL i\leq  L\log L$, 
and  $\log|\g|\leq 2\log L$, taking either $h=i$ or $h=i-1$, with $i\geq n+1$ one has that
$e^{-4\b h}|S_i|\leq L(\log L)^{1-4\b}$ and $e^{-4\b h}\log|\g|\leq 2(\log L)^{1-4\b}$, and therefore
 \begin{align}\label{lb505}
\pi_{S_i,\D_i^+,\D_i^-}^+(\sC_{\g,h}) \leq 
\exp{\big(-(\b-1)|\g| + L
\big)},
\end{align}
as soon as $\b$ and $L$ are large enough. If $\g$ is required to surround $\L_{L-3\ell_i+i}$, then necessarily $|\g|\geq 2L$. Let $p_i$ denote the $\pi_{S_i,\D_i^+,\D_i^-}^+$-probability that there exists either an $(i-1)$-contour or an $i$-contour $\g$
inside $S_i$ and surrounding $\L_{L-3\ell_i+i}$. 
Summing over $\g\subset S_i$ with $|\g|\geq 2L$ in \eqref{lb505}, one finds that for $\b$ large enough,
$p_i\leq e^{-L}$. From \eqref{lb4}, using a union bound and the fact that $Ne^{-L}\leq 1/2$, it follows that 
\begin{align}\label{lb6}
&  \bbP_{\L_L}( \eta_{\L_L}\geq 0\,;\,\cap_{k=1}^n\sC_{\g^*_{k},k}\,;\,\cap_{j=n+1}^N\sC_{\g_{j},j}\,;\,\cup_{i=n+1}^NF_i)\nonumber \\ & \qquad \qquad \leq \sum_{i=n+1}^N p_i \, \bbP_{\L_L}( \eta_{\L_L}\geq 0\,;\,\cap_{k=1}^n\sC_{\g^*_{k},k}\,;\,\cap_{j=n+1}^N\sC_{\g_{j},j})\nonumber\\
 &\qquad \qquad \qquad \qquad  \leq \tfrac12\, \bbP_{\L_L}( \eta_{\L_L}\geq 0\,;\,\cap_{k=1}^n\sC_{\g^*_{k},k}\,;\,\cap_{j=n+1}^N\sC_{\g_{j},j}).
 \end{align}
\end{proof}
Thanks to Lemma \ref{lemlb} the lower bound in Theorem \ref{th:main} follows if we prove
that 
\begin{align}\label{lb7}
\!\!\!\!\!\!\!\! \sum_{\g_{n+1}\in\cC_{n+1},\dots,\g_N\in\cC_N}\!\!\!\!\!\!\!\!\!\!\!\! 
&  \bbP_{\L_L}( \eta_{\L_L}\geq 0\,;\,\cap_{k=1}^n\sC_{\g^*_{k},k}\,;\,\cap_{j=n+1}^N\sC_{\g_{j},j})  \nonumber \\
  & \qquad\qquad \quad \geq \exp{\big(-8\b\t_\b(0)NL(1+o(1))\big)},
 \end{align}
 for any fixed choice of $\g^*_k\in\cC_k$, $k=1,\dots,n$, with $n=\lfloor\log \log L\rfloor$. 
 To prove \eqref{lb7} we start by observing that by the FKG inequality
 one has
  \begin{align}\label{lb50}
\frac{Z_{S_i,\D_i^+,\D_i^-}^+}{Z_{S_i,\D_i^+,\D_i^-}} =
\pi_{S_i,\D_i^+,\D_i^-}(\eta(x)\geq 0,\,\forall x\in S_{i})\geq 
\prod_{x\in S_{i}}\pi_{S_i,\D_i^+,\D_i^-}(\eta(x)\geq 0)\,,
\end{align}
where $Z_{S_i,\D_i^+,\D_i^-}^+$ is defined above \eqref{lb4}, $Z_{S_i,\D_i^+,\D_i^-}$ is as in Section \ref{cl_exp}, and $\pi_{S_i,\D_i^+,\D_i^-}$ denotes the probability measure associated to the partition function $Z_{S_i,\D_i^+,\D_i^-}$. From \cite[Proposition 3.9]{CLMST} one has that $\pi_{S_i,\D_i^+,\D_i^-}(\eta(x)\geq 0)\geq 1-Ce^{-4\b (i-1)}$ for any $x\in S_i$. Using $1-x\geq e^{-2x}$ for $0\leq x\leq 1/2$, one has
   \begin{align}\label{lb51}
\frac{Z_{S_i,\D_i^+,\D_i^-}^+}{Z_{S_i,\D_i^+,\D_i^-}} 
\geq \exp{\big(-2C|S_i|e^{-4\b (i-1)}\big)}\,.
\end{align}
Therefore, in \eqref{lb4} we can estimate
\begin{align}\label{lb40}
&\bbP_{\L_L}( \eta_{\L_L}\geq 0\,;\,\cap_{k=1}^n\sC_{\g^*_{k},k}\,;\,\cap_{j=n+1}^N\sC_{\g_{j},j})\nonumber \\ & \qquad 
\geq \textstyle \exp{\big(-\b\sum_{i=1}^n|\g^*_i|-\b\sum_{i=n+1}^N|\g_i|-2C\sum_{i=1}^{N+1}|S_i|e^{-4\b (i-1)}\big)}
\frac{\prod_{i=1}^{N+1}Z_{S_i,\D_i^+,\D_i^-}}{Z_\L}.
\end{align}
Expanding as in \eqref{ub2a1} one obtains
\begin{align}\label{lb52}
\frac
{\prod_{i=1}^{N+1}Z_{S_i,\D_i^+,\D_i^-}}{Z_\L}
=\textstyle \exp{\big( 
 \Psi_\L(\g^*_1,\dots,\g^*_n,\g_{n+1},\dots,\g_N)\big)},
\end{align}
where $\Psi$ is given in \eqref{ub211}.
 Estimating $|S_i|\leq Ci L$ one finds 
 \begin{align}\label{lb53}
\sum_{i=1}^{N} |S_i|e^{-4\b (i-1)}\leq CL.
 \end{align}
 On the other hand, the term $|S_{N+1}|e^{-4\b N} = |\L_{\g_N}|e^{-4\b H(L)}$ satisfies
 \begin{align}\label{lb533}
|S_{N+1}|e^{-4\b N}\leq L^2 e^{-4\b H(L)} \leq 
CL,
 \end{align}
 where we use $e^{-4\b H(L)}\leq C/L$. Note that it is at this point of the argument that it is crucial to have $N$ as large as $H(L)$.
From \eqref{lb53}-\eqref{lb533} one has $\sum_{i=1}^{N+1} |S_i|e^{-4\b (i-1)}\leq CL$. 
From this bound and \eqref{lb52} we obtain
\begin{align}\label{lb400}
&\bbP_{\L_L}( \eta_{\L_L}\geq 0\,;\,\cap_{k=1}^n\sC_{\g^*_{k},k}\,;\,\cap_{j=n+1}^N\sC_{\g_{j},j})\nonumber \\ & 
\geq \textstyle \exp{\big(-\b\sum_{i=1}^n|\g^*_i|-\b\sum_{i=n+1}^N|\g_i| + \Psi_\L(\g^*_1,\dots,\g^*_n,\g_{n+1},\dots,\g_N) -  CL \big)}.
\end{align}
  Next, we want to show that the interactions among the contours are
  negligible in our setting.
 Let $\psi_\L(\g)$ denote the potential associated to a single contour $\g$ as defined in \eqref{ub1c}.
 \begin{lemma}\label{lemlb1} Take $\b\geq \b_0$. Uniformly in the choice of $\g_1\in \cC_1,\dots,\g_N\in\cC_N$ one has
 \begin{align}\label{lb54}
\textstyle\left|\Psi_\L(\g_1,\dots,\g_N)-\sum_{i=1}^N\psi_\L(\g_i)\right|\leq \sum_{i=1}^{N} |\g_i|e^{-\b i/2}.
 \end{align}
 \end{lemma}
 \begin{proof}
 Notice that any $V\subset \L$ such that $d(V,\g_i)\leq 1$ and $d(V,\g_{i+1})\leq 1$
 must have $d(V)\geq 2i$. Thus the sum of the potentials associated to $V$'s that have $d(V,\g_i)\leq 1$ and 
 are such that $d(V,\g_j)\leq 1$ for some $j\neq i$ contributes at most $|\g_i|e^{-\b i/2}$
 if $\b$ is large enough.  
 \end{proof}
From \eqref{lb400} and Lemma \ref{lemlb1}
one has
\begin{align}\label{lb55}
&\bbP_{\L_L}( \eta_{\L_L}\geq 0\,;\,\cap_{i=1}^n\sC_{\g^*_{i},i}\,;\,\cap_{j=n+1}^N\sC_{\g_{j},j}
)\nonumber \\ & 
\geq \textstyle \exp{\big(-2\b\sum_{i=1}^n|\g^*_i|-\b\sum_{i=n+1}^N|\g_i|(1+e^{-\b i/2}) +\sum_{i=1}^N\psi_\L(\g_i) -CL\big)}
.
\end{align}
For $i=1,\dots,n$, we can use the rough estimates $|\g_i|\leq CL n \leq  C L \log \log L$
and $|\psi_\L(\g_i)| \leq C|\g_i|$ (cf. \eqref{rev1}) to obtain 
\begin{align}\label{lb56} 
\textstyle \exp{\big(-2\b\sum_{i=1}^n|\g_i|+\sum_{i=1}^n\psi_\L(\g_i)\big)}
\geq \exp (-o(L\log L)). 
\end{align}
For $n<i\leq N$ we need the following statement.
\begin{lemma}\label{lemlb2} Uniformly over $i$ such that $n<i\leq N$, one has
 \begin{align}\label{lb57}
\sum_{\g\in\cC_i} \exp{\big(-\b|\g|(1+e^{-\b i/2}) +\psi_\L(\g)\big)}
\geq \exp{\big(-8\b\t_\b(0) L(1+o(1))
\big)}  .
\end{align}
 \end{lemma}
We first conclude the proof of the lower bound in Theorem \ref{th:main}, assuming the estimate of Lemma \ref{lemlb2}. From Lemma \ref{lemlb} and \eqref{lb55}-\eqref{lb56}
we have
\begin{align}\label{lb58}
\bbP_{\L_L}( \eta_{\L_L}\geq0)&\geq \bbP_{\L_L}( \eta_{\L_L}\geq 0; E)
\geq \exp (-o(L\log L))\times \nonumber\\
&\times  
\textstyle\sum_{\g_{n+1}\in\cC_{n+1},\dots,\g_N\in\cC_N}\exp{\big(-\b\sum_{i=n+1}^N|\g_i|(1+e^{-\b i/2}) +\sum_{i=1}^N\psi_\L(\g_i)\big).
}\end{align}
 From Lemma \ref{lemlb2} and using $NL=1/(4\b)L\log L+O(L)$ one has
 \begin{align}\label{lb59}
\bbP_{\L_L}( \eta_{\L_L}\geq0)
\geq\exp{\big(-8\b\t_\b(0) NL(1+o(1))
\big)}  .
\end{align}
This concludes the proof. 

 \begin{proof}[Proof of Lemma \ref{lemlb2}]
 First observe that $\g\in\cC_i$ implies $|\g|\leq |S_i|\leq L\log L$ and therefore
 for $i\geq \log\log L$ and $\b\geq \b_0$ one has
 $$|\g|e^{-\b i/2} = o(L).$$
Next, observe that we may safely replace $\psi_\L(\g)$ in \eqref{lb57} by the quantity $\psi_\infty(\g)$ (see the end of Section \ref{wedding}). Indeed, any connected set $V$ that touches both $\cU_i$ and $\partial \L$ must have $d(V)\geq \tfrac12(\log \log L)^2$. 
Thus, we have to show that 
 \begin{align}\label{lb60}
\sum_{\g\in\cC_i} \exp{\big(-\b|\g| +\psi_\infty(\g)\big)}
\geq \exp{\big(-8\b\t_\b(0) L(1+o(1))
\big)}  .
\end{align}
To prove \eqref{lb60} we fix $i$ and partition the set $\cU_i$ into rectangles $R_j$, $j=1,\dots,m$,
with height $i$ and basis $i^{2-\e}$, so that  there are $m\sim
8Li^{-2+\e}$ such rectangles, see Figure \ref{fig:05}. For simplicity, let us assume that the partitioning is exact so that $\cU_i$ is the union of the $R_j$'s plus four squares at the corners as in Figure \ref{fig:05}. The modifications in the general case are straightforward. 
\begin{figure}[htb]
        \centering
 \begin{overpic}[scale=0.47]{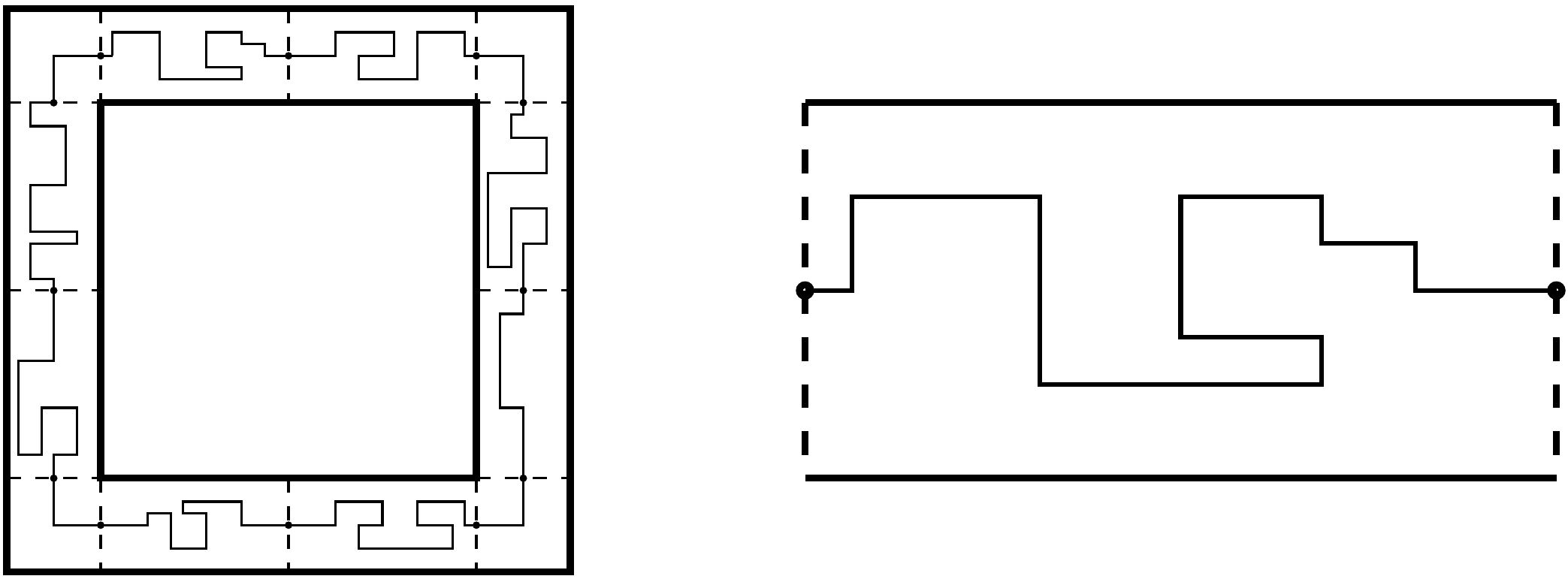}
\put(46,18) {$x_j$}
\put(101.3,18) {$y_j$}
\put(70,16) {$\hat \g_j$}
\end{overpic}
        \caption{
        The partition of $\cU_i$ into rectangles $R_j$, $j=1,\dots, m$ (left). A single path $\hat \g_j:x_j\to y_j$ inside the rectangle $R_j$ (right).  
          }\label{fig:05}
\end{figure}

We fix for every rectangle $R_j$ the points $x_j$ and $y_j$ that are
the midpoints of the two shorter side. Consider an open contour
$\hat\g_j$ connecting $x_j$ to $y_j$ which is entirely contained in
$R_j$ (see Figure \ref{fig:05}).  For technical reasons it is convenient to consider a closed path $\hat \g$ that agrees with $\hat\g_j$ on $R_j$. The latter is defined as follows. Let $\hat \g $ be the closed contour
contained in $\cU_i$ which coincides with $\hat\g_j$ inside $R_j$,
it is given by straight segments in all other rectangles $R_k$, $k\neq j$, and by a straight right angle shape at each of the four corner squares. Then
we define $\psi_\infty(\hat\g_j)$ as $\psi_\infty(\hat \g)$ (see text after \eqref{rev1}) but with the restriction to those sets $V$ which have
distance from $\hat\g_j$ at most $1$.
It follows from  \cite[Sections 4.12 and 4.15]{DKS} that for a fixed index $j$ one has, for $i$ large:
  \begin{align}\label{lb61}
\sum_{ \hat \g_j\,:\;x_j\to y_j,\;\hat \g_j\subset R_j} 
\exp{\big(-\b| \hat \g_j| +\psi_\infty( \hat\g_j)\big)}
\geq \exp{\big(-\b\t_\b(0) i^{2-\e}(1+o(1))
\big)}.
\end{align}
The point is that the height $i$ of the rectangles $R_j$ is much
larger than the typical vertical fluctuation $i^{1-\gep/2}$ of paths
$\hat\gamma_j$, so the restriction to be in $R_j$ is not modifying the
partition function significantly.
 
Suppose now that $\g\in\cC_i$ is a contour 
 passing through all the points $x_j,y_j
$ that can be written as the composition of $\hat\g_1,\dots,\hat\g_m$ where $ \hat\g_j$ is as in the sum above, and assume that it has some prescribed shape at the four corners of $\cU_i$, e.g.\ a right angle form as in  
Figure \ref{fig:05}.  
Then it is immediate to check that 
$|\g|\leq \sum_{j=1}^m|\hat\g_j| + O(i) $, and 
$$\psi_\infty(\g) - \sum_{i=1}^m\psi_\infty(\hat \g_j) = O(i m)\,.$$
The latter estimate holds thanks to the decay properties of the
potentials, so that the mutual interaction between  $\hat\gamma_j$ and
$\hat\gamma_{j-1}$ is $O(i)$ uniformly in $j=1,\dots, m$.
Thus, by restricting the sum in \eqref{lb60} to contours as in \eqref{lb61}
one obtains
\begin{align}\label{lb62}
\sum_{\g\in\cC_i} \exp{\big(-\b|\g| +\psi_\infty(\g)\big)}
\geq \exp{\big(-\b\t_\b(0) m \,i^{2-\e}(1+o(1))
\big)}  .
\end{align}
Since $m\sim 8Li^{-2+\e}$, the desired estimate follows. 
\end{proof}


\section{A monotonicity property of the SOS model}\label{monostate}
Recall the staircase ensemble defined in Section \ref{stair} with partition function
$$\cZ(a_1,\dots,a_n; \,b_1,\dots,\, b_n; L),$$
as defined in Lemma \ref{lem:infvol}.
In this section we establish the following important monotonicity property.
\begin{theorem} 
\label{th:m1}
There exists $\b_0>0$ such that, for any $\b>\b_0$ and any $L\in \bbN$
  \begin{gather}
    \label{eq:2}
\cZ(a_1,\dots,a_n;\, b_1,\dots,b_n;\, L)\leq \prod_{i=1}^n \cZ(a_i;\, b_i;\, 
  L).  
  \end{gather}
\end{theorem}
The above estimate allows one to control the partition function of $n$ interacting open contours by means of the partition functions of $n$ non-interacting open contours. In particular, Theorem \ref{th:m1} and Lemma \ref{lem:surftens} yield the following corollary.

\begin{corollary}\label{coro_surftens}
Fix $n\in\bbN$, and suppose that as $L\to\infty$ one has $(b_i-a_i)/L\to \l_i \in\bbR$, $i=1,\dots,n$. Then 
\begin{align}\label{coros1}
\limsup_{L\to\infty}\frac1{2L}\log \cZ(a_1,\dots,a_n;\, b_1,\dots,b_n;\, L) \leq -\b\sum_{i=1}^n\frac{\t_\b(\theta_i)}{\cos(\theta_i)}
\end{align}
where $\theta_i=\tan^{-1}(\l_i)$. 
\end{corollary}
The proof of Theorem \ref{th:m1} is based on the following key lemma. 
\begin{lemma}
\label{lem:1}
Given $\{a_i,b_i\}_{i=1}^n$,  let $\{a'_i,b'_i\}_{i=1}^n$
  be defined by
\[
a'_i=a_i,\
b'_i=b_i\,,\quad i=1,\dots,n-1;\quad a'_n=a_n+1,\ b'_n=b_n+1.
\] 
Then 
 \[
\cZ(a_1,\dots,a_n;\ b_1,\dots,b_n;\ 
  L)\leq \cZ(a'_1,\dots,\,\,a'_n;\ b'_1,\dots,\, b'_n;\ 
  L).
\] 
\end{lemma}
\begin{proof}[Proof of  Lemma \ref{lem:1}]
Set $\L:= \L_{L,M}$ for some large fixed $M>\max\{a_n,b_n,-a_1,-b_1\}$. Let
$\t,\t'$ be the SOS boundary conditions associated to
$\{a_i,b_i\}_{i=1}^n$ and $\{a'_i,b'_i\}_{i=1}^n$ according to
\eqref{stairtau}.  Given $s\in [0,1]$ consider the auxiliary boundary condition 
$\tau_s: \partial \L\mapsto \bbR$ defined by 
\begin{equation*}
 \tau_s(x_1,x_2)=
 \begin{cases} n-1+s & \text{ if $(x_1,x_2)=(-L-1,a_n)$ or $(x_1,x_2)=(L+1,b_n)$}; \\
   \tau'(x_1,x_2)
   &\text{ otherwise}.
 \end{cases}
\end{equation*}
Next, we  consider the partition function $Z_\L^{\tau_s}$ associated to $\t_s$
(strictly speaking we have only defined the model for integer valued boundary condition, but it is straightforward to extend it to the real valued case).  
Notice that $\t_s = s\t+(1-s)\t'$. We shall see that 
$Z_\L^{\tau_s}$ is differentiable w.r.t.\ $s\in[0,1]$ so that
\begin{equation}\label{diffz}
Z_\L^\t-Z_\L^{\t'}= \int_0^1 ds \, \frac{d}{ds}Z_\L^{\tau_s}.
\end{equation}
In order to compute the above derivative we proceed as follows. 
Define the points
$z=(-(L+1),a_n),w=(-L,a_n)$ and  $z'=(L+1,b_n),\
w'=(L,b_n)$, so that $w$ (resp.\ $w'$) is the nearest neighbor of $z$ (resp.\ $z'$) in $\L$, see Figure \ref{fig:001}.
\begin{figure}[htb]
        \centering
 \begin{overpic}[scale=0.3]{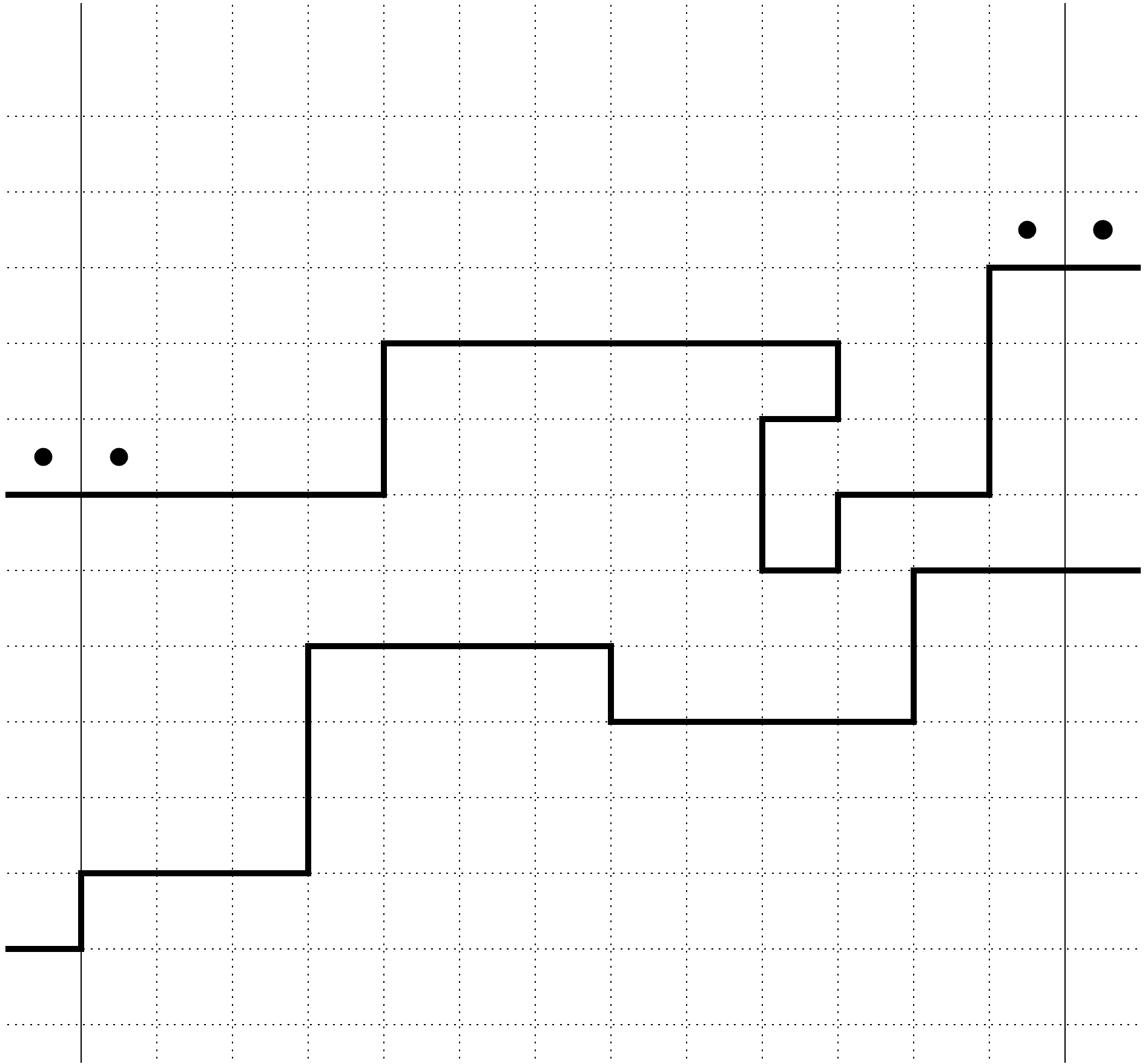}
\put(94.5,77) {$z'$}
\put(2,57){$z$}
\put(8,57){$w$}
\put(86.5,77) {$w'$}
\put(2,2.3){$0$}
\put(95,30){$0$}
\put(2,34.5){$1$}
\put(95,54.5){$1$}
\put(2,71){$2$}
\put(95,87){$2$}
\end{overpic}
        \caption{A sketch of the staircase boundary condition with $n=2$ steps as seen from
               above, with two open contours and the pairs of vertices appearing in the proof of  Lemma \ref{lem:1}: 
          $z=(-L-1,a_2), w=(-L,a_2)$, $z'=(L+1,b_2), w'=(L+1,b_2)$.}\label{fig:001}
\end{figure}

Let $\cB_\L^*=\cB_\L\setminus\{wz,w'z'\}$ denote all bonds with at least one vertex in $\L$
with the exception of the two bonds $wz$ and $w'z'$.
Define the energy function $\cH_\L^{\t,*}(\eta),\ \eta\in \O_\L^{\t_s}$ by
\[
\cH_\L^{\t,*}(\eta)= \sum_{xy\in\cB_\L^*}|\eta(x)-\eta(y)| + \phi(\eta(w)) + \phi(\eta(w')),
\]
where 
\[
\phi(h)= (h-n){\mathds 1}_{\{h\geq n\}} + (n-1-h){\mathds 1}_{\{h\leq n-1\}},\quad
h\in \bbZ.
\]
Since the bonds $wz$ and $w'z'$ are not included in the above sum, we see that $\cH_\L^{\t,*}(\eta)$ does not depend on the parameter $s$. Let also
\[
F_{s,n}(\eta):= \exp\left(-\b\bigl[(1-s)({\mathds
         1}_{\{\eta(w)\geq n\}}+{\mathds
         1}_{\{\eta(w')\geq n\}})+s({\mathds
         1}_{\{\eta(w)\leq n-1\}}+{\mathds
         1}_{\{\eta(w')\leq n-1\}})\bigr]\right). 
\]
Define the partition function $\Xi^{\t,*}_\L= \sum_{\eta\in\O_\L^{\t_s}}\exp(-\b \cH_\L^{\t,*}(\eta))$, and the Gibbs measure $$\pi_\L^{\t,*}(\eta)=(\Xi^{\t,*}_\L)^{-1} \exp(-\b \cH_\L^{\t,*}(\eta)),$$ $\eta\in\O_\L^{\t_s}$. 
It is not hard to check that 
\[
     Z_\L^{\tau_s}= \Xi^{\t,*}_\L \, \pi_\L^{\t,*}\left(F_{s,n}\right).
\]
Using the above expression for $Z_\L^{\tau_s}$ we get
\begin{align*}
\frac{d}{ds} Z_\L^{\tau_s}=  \Xi^{\t,*}_\L \,\pi_\L^{\t,*}\left(\frac{d}{ds} F_{s,n}\right)=
\b\, \Xi^{\t,*}_\L \,\pi_\L^{\t,*}\left(G_s\right),
\end{align*}
where, for any $s\in [0,1]$, we define  
\[
G_{s,n}(\eta):= F_{s,n}(\eta)\left({\mathds
         1}_{\{\eta(w)\geq n\}}+{\mathds
         1}_{\{\eta(w')\geq n\}} - {\mathds
         1}_{\{\eta(w)\leq n-1\}}-{\mathds
         1}_{\{\eta(w')\leq n-1\}}\right).
\]
The function $G_{s,n}$ takes values in $\{-2e^{-2\b s},0,2e^{-2\b(1-s)}\}$ and is easily seen to be  increasing in the configuration $\eta$. 
Therefore, if we raise to height $n-1$ the value of $\t$
on those boundary vertices where it was at most $n-1$ and we denote by $\hat \t$ the resulting boundary
condition, from the FKG inequality we get that 
\begin{gather*}
\pi_\L^{\t,*}\left(G_{s,n}\right)  \leq 
\pi_\L^{\hat\t,*}\left(G_{s,n} \right).
\end{gather*}
The validity of the FKG inequality follows from lattice condition \eqref{holley} for the measure  
$\pi_\L^{\t,*}$, which can be verified directly.
The boundary height $\hat\t$ has now a single step from level $n-1$ to level $n$. Using vertical translation invariance we can now safely replace the
height of $\hat \t$ by $0,1$ instead of $n-1,n$. Finally, 
since $G_{s,1}$ is a bounded local function, we can take the limit
$M\to \infty$ in \eqref{diffz} and get that  
\begin{gather*}
  \cZ(a_1,\dots,a_n;\ b_1,\dots,b_n;\ 
  L)- \cZ(a_1,\dots,a_n+1;\ b_1,\dots,b_n+1;\ L)\\
\leq \b\left(\lim_{M\to \infty}\frac{\Xi^{\t,*}_\L}{Z_\L}\right) \int_0^1
ds\, \pi^{\hat \t,*}_\infty\left(G_{s,1} \right),
\end{gather*}
where $\pi^{\hat \t,*}_\infty(\cdot)$ denotes the weak limit as $M\to\infty$ of $\pi_\L^{\hat\t,*}$, that is the Gibbs measure on $\L_{L,\infty}=[-L,L]\times \bbZ$ with boundary
condition at height 1 at the vertices  $x=(x_1,x_2)$ with
either $x_1=-(L+1)$ and
$x_2\geq a_n+1$ or $x_1=L+1$ and
$x_2\geq b_n+1$; the boundary height is unspecified at the vertices
$z,z'$
(this simply means that the terms corresponding to bonds $wz$ and $w'z'$ do not appear in the interaction) and otherwise it is equal to zero. 
The existence of the limits mentioned above can be proved again from the cluster expansion representation as in Lemma \ref{lem:infvol}. 
By symmetry one has that 
\[
\pi_\infty^{\hat \tau,*}(\eta(w)\geq 1;\, \eta(w')\geq 1)=
\pi_\infty^{\hat \tau,*}(\eta(w)\leq 0;\,  \eta(w')\leq 0), 
\]
so that
\[
\pi_\infty^{\hat \tau,*}\left(G_{s,1} \right)=-\pi_\infty^{\hat \tau,*}\left(G_{1-s,1} \right)  \quad 
\text{ and }\quad \int_0^1ds\, \pi_\infty^{\hat \tau,*}\left(G_{s,1} \right)=0.
\]
In conclusion
\[
  \cZ(a_1,\dots,a_n;\, b_1,\dots,b_n;\, 
  L)\leq \cZ(a'_1,\dots,a'_n;\, b'_1,\dots,b'_n;\, L)
\]
and the lemma is proved.\end{proof}

We can now complete the proof of Theorem \ref{th:m1}. By iterating Lemma
\ref{lem:1} arbitrarily many times, we have that 
\[
   \cZ(a_1,\dots,a_n;\, b_1,\dots,b_n;\, L)\leq \lim_{k\to \infty}
  \cZ(a_1,\dots,a_{n-1},a_n+k;\, b_1,\dots,b_{n-1},b_n+k;\, 
  L).
\]
On the other hand, using the explicit representation \eqref{ub4a1} 
together
with the rough bound  \eqref{rev4} to control the large deviations of the
$n$-th contour $\g_n$, we have
that
\begin{gather*}
\lim_{k\to \infty}
  \cZ(a_1,\dots,a_{n-1},a_n+k;\, b_1,\dots,b_{n-1},b_n+k;\, 
  L)\\= 
  \cZ(a_1,\dots,a_{n-1};\, b_1,\dots,b_{n-1};\, L)\cZ(a_n;\ b_n;\ L).
\end{gather*}
In conclusion, we have factorized out the contribution of the $n$-th
contour. By repeating the above reasoning for $(a_{n-1},b_{n-1}),\,(a_{n-2},b_{n-2})\dots
,(a_{2},b_{2})$ we finally get \eqref{eq:2}.



\section{Upper bound}\label{ub}
If we prove the upper bound for $\bbP$ in \eqref{LDsos}, then we can obtain the upper bound for 
$\bbP_{\L_L}$ by using \eqref{posp1} and Lemma \ref{lem:bd}. From now on  we concentrate on proving the upper bound for $\bbP$.

For any event $A$, note that 
\begin{align}\label{ub1}
\bbP( \eta_{\L_L}\geq 0) \leq  
\frac{\bbP(A)}{\bbP(A\tc \eta_{\L_L}\geq 0)}.
\end{align} 
Indeed, \eqref{ub1} is obtained by multipling by $\bbP( \eta_{\L_L}\geq 0)$ both sides of the obvious inequality $1\leq \bbP(A)/
\bbP(A,\eta_{\L_L}\geq 0)$.

For any $\d>0$ and $K>0$, define $A=A(\d,K)$, as the event that there exists a lattice circuit $\cC$ surrounding $\L':=\L_{(1-\d)L}$
such that $\eta(x)\geq H(L) - K$, for all $x\in\cC$, where as usual $H(L)=\lfloor\tfrac1{4\b} \log L\rfloor$. 
\begin{proposition}\label{propoad1}
For any $\d>0$, there exists a constant $K>0$ such that 
\begin{align}\label{adub2}
\lim_{L\to\infty}\bbP(A(\d,K)\tc \eta_{\L_L}\geq 0)= 1.
\end{align}
\end{proposition}
\begin{proof}
Let $\partial_* \L_L$ denote the internal boundary of $\L_L$. 
Observe that $A(\d,K)$ is monotone increasing so that by the FKG inequality $$\bbP(A(\d,K)\tc \eta_{\L_L}\geq 0)\geq \bbP(A(\d,K)\tc \eta_{\L_L}\geq 0, \,{\eta_{\partial_*\L_{L}}}=0).$$
Therefore, the proposition follows once we know that for some $K=K(\d)$ one has 
\begin{align}\label{adub02}
\lim_{L\to\infty}\bbP(A(\d,K)\tc \eta_{\L_L}\geq 0, \,{\eta_{\partial_*\L_{L}}}=0)= 1.
\end{align}
Under the conditioning $\eta_{\L_L}\geq 0,
\,{\eta_{\partial_*\L_{L}}}=0$, one has an SOS interface in $\L_{L-1}$
with a wall at height zero and zero boundary conditions. The result of \cite[Theorem 2]{CLMST2}
implies that with probability converging to $1$, within $\L_{L-1}$, there exists an $h$-contour surrounding $\L'=\L_{(1-\d)L}$, for all $h\leq H(L) -K$ as soon as $K$ is a sufficiently large constant depending on $\d$. 
This implies \eqref{adub02}.
\end{proof}

It follows that to prove the upper bound in \eqref{LDsos} it is sufficient to establish:
\begin{proposition}\label{propoad}
For any $\d>0$, for any $K>0$, one has
\begin{align}\label{ub2}
\limsup_{L\to\infty}\frac1{2L\log L}\log \bbP( A(\d,K)) \leq  -\t_\b(0)(1-\d).
\end{align}
\end{proposition}

\subsection{Proof of Proposition \ref{propoad}} The first observation
is that we may impose zero boundary conditions outside a very large
set, e.g.\ $\L_M$ with $M\gg L^2$, and therefore we may consider
$\wt\bbP:=\bbP_{\L_M}$ instead of $\bbP$ in \eqref{ub2}. 
The reason is that the probability that there is a contour surrounding
$\L'$ and not contained in, say, $\L_{L^2}$ is a negligible
$O(\exp(-L^2))$, as one can check easily using a rough estimate as in \eqref{rev1}. Then, $A(\d,K)$ can be
considered as a local event (localized in $\L_{L^2}$) and by
definition of thermodynamic limit one can approximate arbitrarily well
$\bbP( A(\d,K)) $ by $\wt\bbP( A(\d,K)) $, if $M$ is sufficiently large.
 
The event $A(\d,K)$ implies that for each $h=1,\dots,N:=H(L) - K$ there exists (at least) one  $h$-contour  surrounding $\L'$.
Therefore, there must exist $\L_M\supset \g_1\supset\cdots \supset\g_N\supset \L'$ such that $\g_h$ is an $h$-contour:
\begin{align}\label{ub20}
\wt \bbP( A(\d,K)) \leq \sum_{  \g_1\supset\cdots \supset\g_N\supset \L'} \wt \bbP\big( \cap_{i=1}^N\sC_{\g_i,i}\big).
\end{align}
Here we use the notation $\L_M\supset \g_1\supset\cdots \supset\g_N\supset \L'$ when the contours satisfy $\L_M\supset \L_{\g_1}\supset\cdots \supset\L_{\g_N}\supset \L'$.

For a fixed choice of $\g_1\supset\cdots \supset\g_N$ the above probability is computed in \eqref{ub2a1}:
\begin{align}\label{ub21}
\wt \bbP\big( \cap_{i=1}^N\sC_{\g_i,i}\big)= 
\textstyle \exp{\big(-\b\sum_{i=1}^N|\g_i| + 
 \Psi_{\L_M}(\g_1,\dots,\g_N)\big)}.
\end{align}

To deal with the  summation in \eqref{ub20}  we consider a decomposition of each contour into four ``irreducible" pieces, which will be responsible for the main contributions, plus some negligible corner terms. 

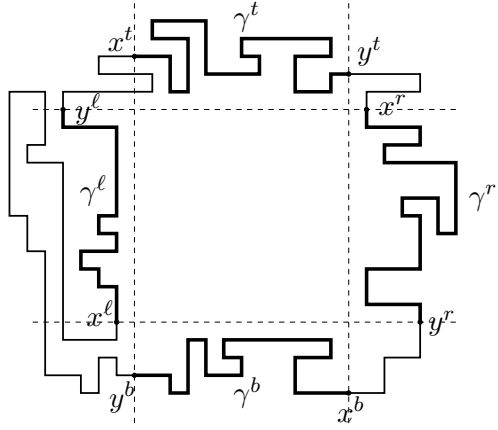
\begin{figure}[htb]
        \centering
 \begin{overpic}[scale=0.37]{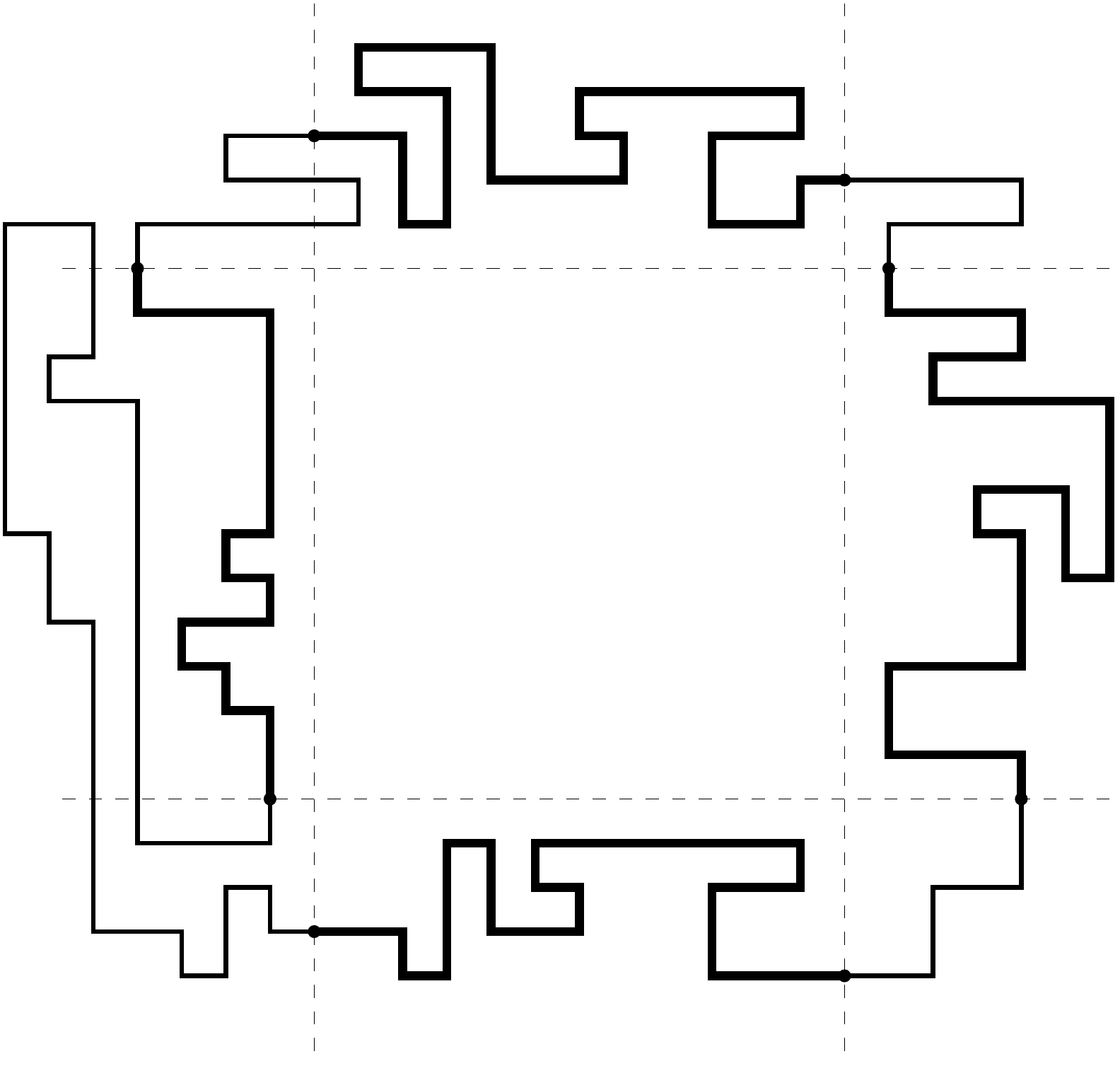}
\put(22.2,84.3) {$x^t$}
\put(77.3,82.2) {$y^t$}
\put(50,89.8) {$\g^t$}
\put(16,51) {$\g^\ell$}
\put(15,69.2) {$y^\ell$}
\put(17.5,23) {$x^\ell$}
\put(22.5,5) {$y^b$}
\put(72.9,1.6) {$x^b$}
\put(50,6) {$\g^b$}
\put(93,22) {$y^r$}
\put(82,70) {$x^r$}
\put(102,50) {$\g^r$}
\end{overpic}
        \caption{
        Example of a contour $\g$ surrounding the square $\L'$. The irreducible components of $\g$ are the thicker paths.
          }\label{fig:2}
\end{figure}

Let $\cS_v$ and $\cS_h$ denote, respectively, the vertical and horizontal infinite strips obtained by prolonging the sides of the square $\L'$:
\begin{align*}
\cS_v&=\{x=(x_1,x_2)\in\bbZ^2:\;|x_1|\leq (1-\d)L\}\,,\;\\
\cS_h&=\{x=(x_1,x_2)\in\bbZ^2:\;|x_2|\leq (1-\d)L\}.
\end{align*}
Let $\cS_v^t$, resp. $\cS_v^b$, denote the top, resp.\ bottom part of $\cS_v$, i.e.\ the part that lies above, resp. below, the square $\L'$. Similarly, let $\cS_h^\ell$, resp.\ $\cS_h^r$, denote the portion of $\cS_h$ to the left, resp. to the right, of the square $\L'$.

We now define the irreducible components of a fixed  contour $\g$ containing $\L'$. 
Consider the portion of $\g$ that intersects $ \cS_v^t$. 
This must contain at least one {\em crossing}, defined as an open contour connecting the opposite vertical sides of $\cS_v^t$ that is fully contained in the interior of $\cS_v^t$. Let $\g^t$ denote the most internal crossing, i.e.\ the one that lies closest to the square $\L'$. We repeat the same construction in the strips $\cS_h^\ell, \cS_v^b$ and $S_h^r$, to define $\g^\ell, \g^b$ and $\g^r$ as the most internal crossings. We say that $\g^u$, $u\in\{t,\ell,b,r\}$, 
form the {\em irreducible components} of the contour $\g$. 
We call $x^u,y^u$ the endpoints of $\g^u$, with $x^u$ coming after $y^u$ if $\g^u$ is given a counter clockwise orientation. 
See Figure \ref{fig:2}.
It is easy to convince oneself that any contour 
containing the square $\L'$, such that 
its irreducible components coincide with the given $\g^t,\g^\ell, \g^b,\g^r$, must have the following property:
If we travel along $\g^t$ 
in the direction $y^t \to x^t$, and then follow the contour, 
the irreducible components we meet are, in order:  $\g^\ell$ 
in the direction $y^\ell \to x^\ell$,  then $\g^b$ in the direction $y^b \to x^b$, then $\g^r$ in the direction $y^r \to x^r$, and finally again $\g^t$ 
in the direction $y^t \to x^t$.
Thus we can write any $\g$ with given irreducible components $\g^t,\g^\ell, \g^b,\g^r$ as the composition
\begin{align}\label{ub22}
&\g= \g^t\circ \h^{t,\ell}\circ \g^\ell\circ \h^{\ell,b}\circ\g^b\circ \h^{b,r}\circ\g^r\circ \h^{r,t},
\end{align}
where $\h^{u,v}$ denotes a path connecting $x^u$ and $y^v$ for $u,v\in\{t,\ell,b,r\}$.

Let 
$\g_1,\dotsm\g_N$ denote a collection of nested contours as in \eqref{ub20}. We write $\g_i^u$ for the corresponding irreducible components, and $\eta_i^{u,v}$ for the remaining components.
Clearly, by applying the decomposition \eqref{ub22} for each $i$, one has 
\begin{align}\label{ub23}
&
\textstyle|\g_i|=  |\h_i^{t,\ell}|+|\h_i^{\ell,b}|+|\h_i^{b,r}|+| \h_i^{r,t}| + \sum_u|\g_i^u|,
\end{align}
where the sum ranges over $u\in\{t,\ell,b,r\}$.

Next, we want to decouple the four irreducible pieces, by writing  
$\Psi_{\L_M}(\g_1,\dots,\g_N)$ as the sum of a main term $\sum_u\Psi_u(\g^u_1,\dots,\g^u_N)$ and a correction term associated to the corner pieces $\eta_i$ and to the interactions between distinct irreducible regions. To this end it will be convenient to enlarge  the strips $\cS_v,\cS_h$ by an amount of order $(\log L)^2$. This will ensure that the expression \eqref{ub21} factorizes (up to lower order terms) into  the product of four distinct pieces which, see  Lemma \ref{prozs} below, can each be reinterpreted as probabilities from the SOS staircase ensemble defined in Section \ref{stair}. 
 To define the potential  $\Psi_u(\g^u_1,\dots,\g^u_N)$ we proceed as follows. 
\begin{figure}[htb]
        \centering
 \begin{overpic}[scale=0.45]{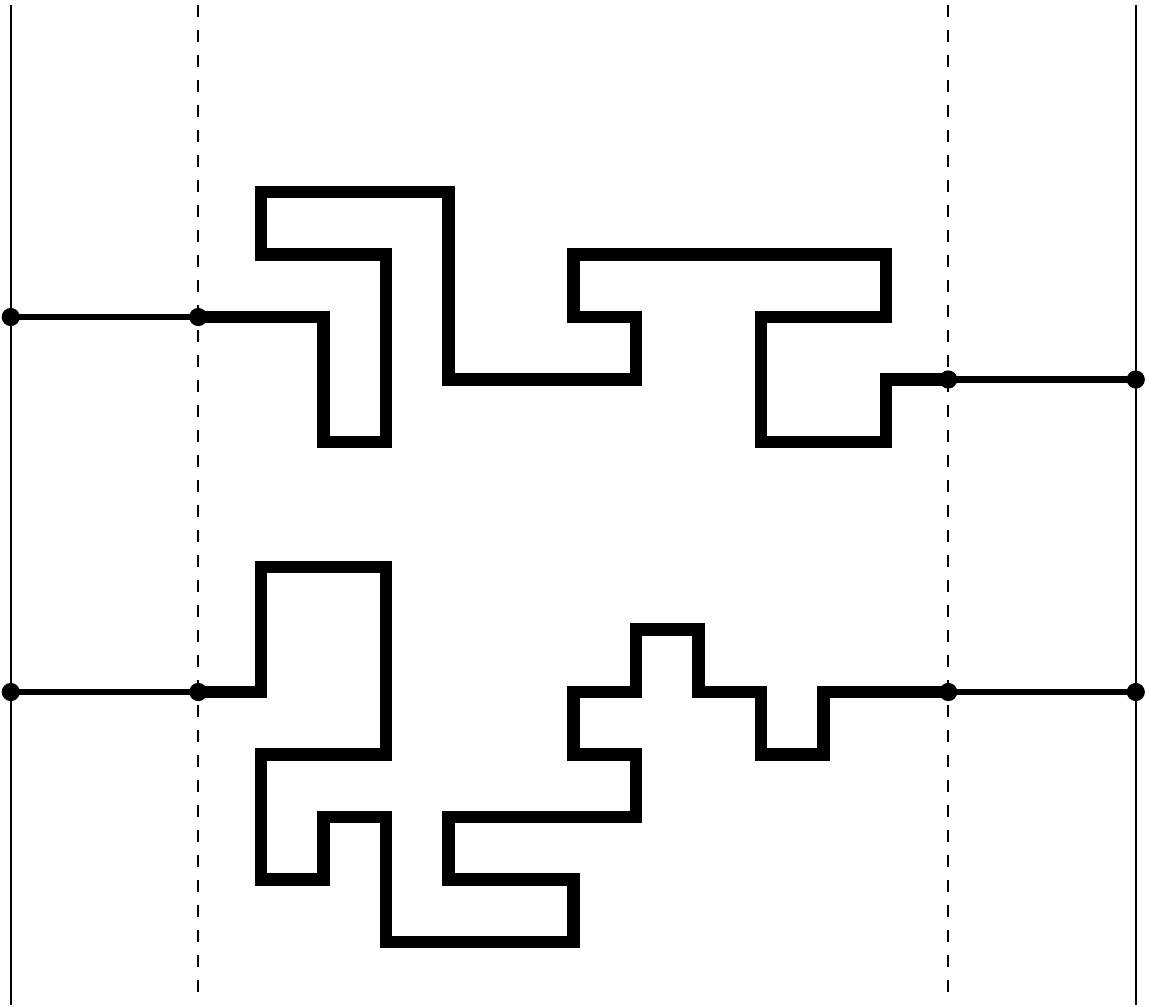}
\put(10,63) {$x_1^t$}
\put(-8,63) {$\hat x_1^t$}
\put(10,31) {$x_2^t$}
\put(-8,31) {$\hat x_2^t$}
\put(39,21) {$\g_2^t$}\put(55,70) {$\g_1^t$}
\put(83,58) {$y_1^t$}\put(101,58) {$\hat y_1^t$}
\put(83,31) {$y_2^t$}\put(101,31) {$\hat y_2^t$}
\end{overpic}
        \caption{
        Picture of two open contours $\hat\g^t_i:\hat y_i^t\to \hat x_i^t$, $i=1,2$     in the enlarged strip $\cS_v'$.  The paths are obtained by adding straight lines to the corresponding irreducible components $\g_i^t$.  }\label{fig:3}
\end{figure}
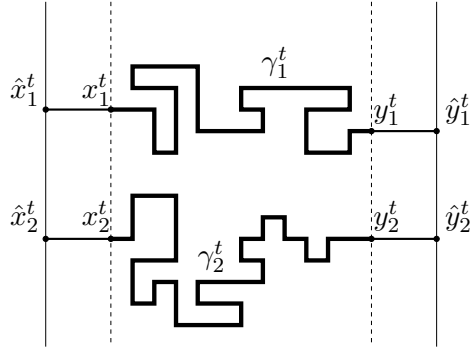

We start with $u=t$.  
Let $\cS_v'$ denote the infinite vertical strip obtained by enlarging the original strip $\cS_v$ by $(\log L)^2$:    $$\cS_v' = \{x\in\bbZ^2:\, d(x,\cS_v)\leq (\log L)^2\}.$$
Let $\hat x^t_i$ denote the point on the left boundary of $\cS_v'$ which has the same vertical coordinate as $x^t_i$ and let $\hat y^t_i$ denote the point on the right boundary  of $\cS_v'$ which has the same  vertical coordinate as $y^t_i$. Let $\hat\g^t_i$ denote the open contour joining $\hat x^t_i$ and $\hat y^t_i$ 
obtained by connecting $\hat x^t_i$ and $x^t_i$ by a straight line, then using $\g^t_i$ from $x^t_i$ to $y^t_i$ and then connecting $y^t_i$ and $\hat y^t_i$ by a straight line; see Figure \ref{fig:3}. 
This defines  a set of ordered, non-crossing paths $\hat\g^t_i$, $i=1,\dots,N$ in the strip $\cS_v'$, 
all staying above the square $\L'$. For a given choice of $\g^t_1,\dots,\g^t_N$, we define the potential:
 \begin{align}\label{ub230}
 \Psi_t(\g^t_1,\dots,\g^t_N) := \Phi_{L',\infty}(\hat\g^t_1,\dots,
\hat\g^t_N)\,, 
\end{align}
where $L'=\lfloor(1-\d)L+ (\log L)^2\rfloor$ is half the width of  the strip $\cS_{v}'$, and $\Phi_{L',\infty}$ is defined in \eqref{ub4a1}.
The potentials $\Psi_u(\g^u_1,\dots,\g^u_N)$,
for $u=\ell,b,r$ are defined in the very same way, with the obvious modifications. 
\begin{lemma}\label{lemmapsi}
Let $\Psi_{\L_M}$ denote the potential from \eqref{ub21}. There exists $\b_0,C>0$ such that: for any choice of $\g_1,\dots,\g_N$ in \eqref{ub20} with 
$\g_1\subset \L_{L^2/2}$, for any $\b\geq \b_0$ one has 
\begin{align}\label{ub231}
& |\Psi_{\L_M}(\g_1,\dots,\g_N)-\sum_u\Psi_u(\g^u_1,\dots,\g^u_N) |\nonumber \\
& \quad \qquad \leq C \sum_{i=1}^{N+1}( |\h_i^{t,\ell}|+|\h_i^{\ell,b}|+|\h_i^{b,r}|+| \h_i^{r,t}|) + C(\log L)^3
\end{align}
\end{lemma}
\begin{proof}
We are going to use the properties of the potentials listed in Lemma \ref{lem:dks}. 
In particular, we use the fact that for $\b$ large enough, for any $\G\subset \bbZ^2$, any $\l>0$ one has
\begin{align}\label{ub2310}
\textstyle
\sumtwo{V\subset \bbZ^2:}{V\cap \G\neq \emptyset,\,d(V)\geq \l}\sup_{{U_+,U_-}}|\varphi_{U_+,U_-}(V)|\leq
 C|\G|\,e^{-\l}
\end{align}
for some constant $C>0$. 
In the potential $\Psi_{\L_M}$ one has a sum over subsets $V\subset \L_M$, while the potential $ \Psi_u$
contains sums over $V$ in the corresponding strips of width $2L'$.  Since we assume $\g_1\subset \L_{L^2/2}$, one has that $d(\g_1,\L_M^c)>L^2/4$ and 
therefore adding all $V$'s which are not contained in $\L_M$ does not change the value of $\Psi_{\L_M}(\g_1,\dots,\g_N)$ by more than a constant. 
Similarly, using the fact that there are $N=O(\log L)$ contours and that $\g^t_i$ is at distance at least $\l=(\log L)^2$ from the complement of $ \cS'_v$, 
when we compute $\Psi_t$, we may remove the constraint that $V\subset \cS'_v$ at the cost of an additive term $O((\log L)^3)$. Indeed,  
separating  the contribution from the straight pieces in $\hat \g^t_i$, and observing that $\max_i |\g^t_i|\leq CL^2 $ (since all contours belong to $\L_M$, with $M=L^2$) one has that the sum over all $V\not\subset \cS'_v$ at distance less than $1$ from $\cup_{i=1}^N\hat\g_i^t$
contributes at most 
$$
C N L^2 e^{-(\log L)^2} + CN(\log L)^2\leq C(\log L)^3 .
$$
The same applies to all $\Psi_u$, $u\in\{t,\ell,b,r\}$.
The same reasoning shows that the sum over all $V$'s such that $V$ intersects 
both $\g^u_i$ and $\g^v_j$, for arbitrary $i,j$ is at most a constant if  $u\neq v$.
It remains to deal with the contribution from all the $V$'s which intersect some corner term $\h_i^{u,v}$. By the rough bound \eqref{ub2310} these can be estimated by $C|\h_i^{u,v}|$. Putting together these facts one arrives at \eqref{ub231}.   \end{proof}

From \eqref{ub21}, if $\g_1\subset \L_{L^2/2}$, then 
 Lemma \ref{lemmapsi} implies for $\b$ large enough: 
\begin{align}\label{ub25}
\wt \bbP\big( \cap_{i=1}^N\sC_{\g_i,i}\big) &\leq 
\textstyle \exp{\big(-\frac12\b\sum_{i=1}^N(|\h_i^{t,\ell}|+|\h_i^{\ell,b}|+|\h_i^{b,r}|+| \h_i^{r,t}|)+ C(\log L)^3\big)}\times \nonumber\\
& \qquad \textstyle
\times\prod_{u}\exp{\big(-\b\sum_{i=1}^N|\g_i^u| + \Psi_u(\g_1^u,\dots,\g_N^u)\big)},
\end{align}
Let us now go back to \eqref{ub20}. Using a very rough bound one can easily obtain 
\begin{align}\label{ub24}
\wt \bbP( \g_1\not \subset \L_{L^2/2} )\leq e^{-L^2}.
\end{align}
Indeed, write the expansion \eqref{ub2a1}
 with only one contour and estimate the decoration term $|\psi_\L(\g_1)|\leq c_\b|\g_1|$, with a constant $c_\b>0$ that vanishes as $\b\to\infty$, and then use a simple Peierls' argument together with the fact that $\g_1\not \subset \L_{L^2/2}$ implies $|\g_1|\geq L^2/2$.  

From \eqref{ub24} and  \eqref{ub25}, summing over all choices of the points 
$$(x,y)= \{(x_i^u,y_i^u),\, i=1,\dots,N; u=t,\ell,b,r\},$$ one has 
that up to the additive error term $e^{-L^2}$, $\wt\bbP(A(\d,K))$ is upper bounded by
\begin{align}\label{ub26} 
\sum_{(x,y)} \Big(\prod_{i=1}^N\Theta(x_i^t,y_i^\ell)\Theta(x_i^\ell,y_i^b)\Theta(x_i^b,y_i^r)\Theta(x_i^r,y_i^t)\Big)\prod_u \cZ_u(x^u,y^u) ,
\end{align}
where 
\begin{align}\label{ub27}
\cZ_u(x^u,y^u) := \sum_{\g_1^u,\dots,\g_N^u}\exp{\textstyle \big(-\b\sum_{i=1}^N|\g_i^u| +\Psi_u(\g_1^u,\dots,\g_N^u)\big)},
\end{align}
and
\begin{align}\label{ub28}
\Theta(x_i^u,y_i^v) := e^{C(\log L)^3}\!\!\!\!\!\sum_{\eta:\,x_i^u\to y_i^v}\exp{(-\textstyle \frac12\b|\eta|)}.
\end{align}
The sum in \eqref{ub27} ranges over all open contours $\g_i^u:y_i^u\to x_i^u$ such that  $\g_{i}^u,\g_{j}^u $ do not cross for $i\neq j$ and such that $\g_{i}^u$ is more internal (closer to $\L'$) than $\g_{j}^u$ for $i>j$. Since we are doing an upper bound,  we may neglect the constraint that $\g_i^u$ does not cross the boundary of $\L'$. The sum  in \eqref{ub28}
ranges over all paths from $x_i^u\to y_i^v$.
The following lemma summarizes the main estimate we need. 
\begin{lemma}\label{prozs}
For any $u$, uniformly in the choice of the points $x^u,y^u$, one has
\begin{align}\label{ub29}
\cZ_u(x^u,y^u) \leq \exp {\big(-2\b \t_\b(0)N L (1-\d) (1+o(1))\big)}  .
\end{align}
\end{lemma}
Let us conclude the proof by assuming the validity of Lemma \ref{prozs}. 
From \eqref{ub28} one has 
$$
\sum_{x_i^u,y_i^v}
\Theta(x_i^u,y_i^v) \leq e^{C(\log L)^3},
$$
for some new constant $C$. Therefore, one has the upper bound 
$$
 \sum_{(x,y)} \Big(\prod_{i=1}^N\Theta(x_i^t,y_i^\ell)\Theta(x_i^\ell,y_i^b)\Theta(x_i^b,y_i^r)\Theta(x_i^r,y_i^t)\Big)
 \leq e^{C(\log L)^4}.
 $$
From \eqref{ub20}-\eqref{ub26}, using the uniform bound \eqref{ub29} for each $u$, one has
\begin{align}\label{ub30}
\wt \bbP( A(\d,K)) \leq e^{C(\log L)^4}\exp {\big(-8\b \t_\b(0)N (1-\d)L  (1+o(1))\big)}.
\end{align}
Since $N=\frac1{4\b}\log L(1+o(1))$ the conclusion \eqref{ub2} follows.  

\subsection{Proof of Lemma \ref{prozs}}\label{mainproof}
The core of the proof is the monotonicity argument of Theorem \ref{th:m1} that allows us to consider each of the $N$ contours separately; see Section \ref{monostate}.
To be able  to apply this we first need to reformulate the problem in terms of $SOS$ contours.   
Without loss of generality we assume that $u=t$. 
Let $\hat x_i^t,\dots,\hat y_i^t$ denote the points on the boundary of $\cS_v'$ as defined before 
\eqref{ub230}, and call $a_{N-i+1}$ the vertical coordinate of $\hat x_i^t$ and $b_{N-i+1}$ the vertical coordinate of $\hat y_i^t$, $i=1,\dots,N$. Let $\cZ(a_1,\dots,a_N;b_1,\dots,b_N; L')$, $L'=(1-\d)L+(\log L)^2$, denote the partition function of the $N$ contours in the strip $\cS_v'$ as defined in Lemma \ref{lem:infvol}.
We claim that 
\begin{align}\label{ub31}
\cZ_t(x^t,y^t)\leq e^{C(\log L)^3} \,
\cZ(a_1,\dots,a_N;b_1,\dots,b_N; L').
\end{align}
Let us first conclude the proof of Lemma \ref{prozs} assuming the validity of the estimate \eqref{ub31}. From \eqref{ub31} and 
Theorem \ref{th:m1} we can bound $\cZ_t(x^t,y^t)$ from above by a product of partition functions of a single contour:
\begin{align}\label{ub32}
\cZ_t(x^t,y^t)\leq e^{C(\log L)^3} \,
\prod_{i=1}^N\cZ(a_i;b_i; L').
\end{align}
The surface tension bound \eqref{eq:surftens2} then implies the desired estimate
\eqref{ub29}.
%

To conclude the  proof of Lemma \ref{prozs}, it remains to prove \eqref{ub31}. 
To this end, observe that by the expansion \eqref{ub4a1}, one has 
\begin{align}\label{ub34}
\cZ(a_1,\dots,a_N;b_1,\dots,b_N; L')
=\sum_{\hat \g_1,\dots,\hat\g_n}\exp{\textstyle \big(-\b\sum_{i=1}^N|\hat\g_i| +\Phi_{L',\infty}(\hat\g_1,\dots,\hat\g_N)\big)},
\end{align}
where the sum ranges over all collections of non-crossing contours $\hat\g_i:\hat x_i^t\to \hat y_i^t$.
Let us restrict this summation to paths of the form $\hat\g_i = \hat\g^t_i$, i.e.\ paths which have a straight line from $\hat x_i^t$ to $x_i^t$, a regular path $\g_i^t:x_i^t\to y_i^t$, and a straight line from $y_i^t\to \hat y_i^t$; see Figure \ref{fig:3}. By summing over the regular parts $\g_i^t$ and using $|\hat\g^t_i|=|\g^t_i|+2(\log L)^2$ one has
\begin{align}\label{ub35}
&\cZ(a_1,\dots,a_N;b_1,\dots,b_N; L')\nonumber \\
&\qquad \geq \sum_{\g^t_1,\dots,\g^t_n}
\exp{\textstyle \big(-\b\sum_{i=1}^N|\g^t_i| +\Phi_{L',\infty}(\hat\g^t_1,\dots,\hat\g^t_N) - 2\b N(\log L)^2}\big),
\end{align}
By the definition \eqref{ub230}, one has $\Phi_{L',\infty}(\hat\g^t_1,\dots,\hat\g^t_N) = \Psi_t(\g^t_1,\dots,\g^t_N)$. Therefore, using $N\leq (4\b)^{-1}\log L$, we conclude
\begin{align}\label{ub36}
\cZ(a_1,\dots,a_N;b_1,\dots,b_N; L')
\geq \cZ_t(x^t,y^t)\,e^{-C(\log L)^3}.
\end{align}
This ends the proof of \eqref{ub31}.


\bibliographystyle{plain}
\bibliography{sos.bib}

\begin{thebibliography}{10}

\bibitem{BDG}
Erwin Bolthausen, Jean-Dominique Deuschel, and Giambattista Giacomin.
\newblock Entropic repulsion and the maximum of the two-dimensional harmonic
  crystal.
\newblock {\em Ann. Probab.}, 29(4):1670--1692, 2001.

\bibitem{BDZ}
Erwin Bolthausen, Jean-Dominique Deuschel, and Ofer Zeitouni.
\newblock Entropic repulsion of the lattice free field.
\newblock {\em Comm. Math. Phys.}, 170(2):417--443, 1995.

\bibitem{BW}
R.~Brandenberger and C.~E. Wayne.
\newblock Decay of correlations in surface models.
\newblock {\em J. Statist. Phys.}, 27(3):425--440, 1982.

\bibitem{Bricetal}
J.~Bricmont, A.~El~Mellouki, and J.~Fr{\"o}hlich.
\newblock Random surfaces in statistical mechanics: roughening, rounding,
  wetting, {$\ldots\,$}.
\newblock {\em J. Statist. Phys.}, 42(5-6):743--798, 1986.

\bibitem{CLMST}
Pietro Caputo, Eyal Lubetzky, Fabio Martinelli, Allan Sly, and Fabio~Lucio
  Toninelli.
\newblock Dynamics of 2+1 dimensional sos surfaces above a wall: slow mixing
  induced by entropic repulsion.
\newblock To appear on Ann. Probab., preprint \texttt{arXiv:1205.6884}, 2012.

\bibitem{CRAS}
Pietro Caputo, Eyal Lubetzky, Fabio Martinelli, Allan Sly, and Fabio~Lucio
  Toninelli.
\newblock The shape of the {(2+1)D} {SOS} surface above a wall.
\newblock {\em C. R. Math. Acad. Sci. Paris}, 350(13-14):703--706, 2012.

\bibitem{CLMST2}
Pietro Caputo, Eyal Lubetzky, Fabio Martinelli, Allan Sly, and Fabio~Lucio
  Toninelli.
\newblock Scaling limit and cube-root fluctuations in sos surfaces above a
  wall.
\newblock To appear on J. Eur. Math. Soc., preprint \texttt{arXiv:1302.6941},
  2013.

\bibitem{Dzero}
Jean-Dominique Deuschel.
\newblock Entropic repulsion of the lattice free field. {II}. {T}he
  {$0$}-boundary case.
\newblock {\em Comm. Math. Phys.}, 181(3):647--665, 1996.

\bibitem{DG}
Jean-Dominique Deuschel and Giambattista Giacomin.
\newblock Entropic repulsion for massless fields.
\newblock {\em Stochastic Process. Appl.}, 89(2):333--354, 2000.

\bibitem{DGI}
Jean-Dominique Deuschel, Giambattista Giacomin, and Dmitry Ioffe.
\newblock Large deviations and concentration properties for {$\nabla\phi$}
  interface models.
\newblock {\em Probab. Theory Related Fields}, 117(1):49--111, 2000.

\bibitem{DKS}
R.~Dobrushin, R.~Koteck{\'y}, and S.~Shlosman.
\newblock {\em Wulff construction}, volume 104 of {\em Translations of
  Mathematical Monographs}.
\newblock American Mathematical Society, Providence, RI, 1992.
\newblock A global shape from local interaction, Translated from the Russian by
  the authors.

\bibitem{FKG}
C.~M. Fortuin, P.~W. Kasteleyn, and J.~Ginibre.
\newblock Correlation inequalities on some partially ordered sets.
\newblock {\em Comm. Math. Phys.}, 22:89--103, 1971.

\bibitem{KotPre}
R.~Koteck{\'y} and D.~Preiss.
\newblock Cluster expansion for abstract polymer models.
\newblock {\em Comm. Math. Phys.}, 103(3):491--498, 1986.

\bibitem{LebMaes}
Joel~L. Lebowitz and Christian Maes.
\newblock The effect of an external field on an interface, entropic repulsion.
\newblock {\em J. Statist. Phys.}, 46(1-2):39--49, 1987.

\bibitem{LMS}
Eyal Lubetzky, Fabio Martinelli, and Allan Sly.
\newblock Harmonic pinnacles in the {D}iscrete {G}aussian model.
\newblock preprint \texttt{arXiv:1405.5241}, 2014.

\bibitem{Yvan_survey}
Yvan Velenik.
\newblock Localization and delocalization of random interfaces.
\newblock {\em Probab. Surv.}, 3:112--169, 2006.

\end{thebibliography}

\end{document}